\documentclass[10pt,reqno]{article}
\usepackage{amssymb, amsmath, amsthm, amsfonts, amscd, epsfig, subcaption}
\usepackage[dvipsnames]{xcolor}
\usepackage[utf8]{inputenc}
\usepackage[english]{babel}
\usepackage{bm}
\usepackage{bbm}
\usepackage{graphicx, epsfig}
\usepackage{geometry,graphicx,pgfplots}
\usepackage{setspace}
\usepackage[export]{adjustbox}
\usepackage[titletoc,toc,title]{appendix}

\usepackage{hyperref}
\usepackage{cleveref}

\usepackage{algorithm}
\usepackage{mathtools}

\usepackage{afterpage}
\usepackage{comment}
\usepackage{cases}

\newtheorem{theorem}{Theorem}[section]

\newtheorem{lemma}[theorem]{Lemma}                                                                                                                                                                                                                                                                             

\newtheorem{definition}{Definition}

\newtheorem{example}{Example}
\newtheorem{remark}{Remark}

\crefname{theorem}{Theorem}{Theorem}

\theoremstyle{definition}

\def\RR{\mathbb{R}}
\def\NN{\mathbb{N}}
\def\CC{\mathbb{C}}
\def\ZZ{\mathbb{Z}}
\def\R{\mathbb{R}}
\def\L{\mathcal{L}}
\def\C{\mathcal{C}}
\def\bu{\mathbf{u}}
\def\bv{\mathbf{v}}

\def\:={\coloneqq}

\def\G{\mathbf{\Gamma}}

\def\Ss{{\mathcal{S}}}

\def\Dd{{\mathcal{D}}}
\def\vp{{\varphi}}
\def\bvp{{\bm{\varphi}}}

\def\Kk{{\mathcal{K}}}
\def\Scal{{\mathcal{S}}}

\def\bpsi{\bm{\psi}}
\def\l{\lambda}

\def\li{\tilde{\lambda}}
\def\m{\mu}

\def\mi{\tilde{\mu}}

\def\f{\mathbf{f}}
\def\g{\mathbf{g}}

\def\F{\mathbf{F}}
\def\H{\mathbf{H}}
\def\p{\partial}
\def\n{\nabla}

\def\a{\alpha}
\def\b{\beta}

\def\<{\left\langle}
\def\>{\right\rangle}

\def\rmi {\mathrm{i}}

\def\tlambda{\tilde{\lambda}}
\def\tmu{\tilde{\mu}}
\newcommand{\wtF}{{\widetilde{F}}}

\newcommand{\Om}{\Omega}

\newcommand{\pd}[2]{\frac {\p #1}{\p #2}}

\def\beq{\begin{equation}}
\def\eeq{\end{equation}}

\newcommand{\ds}{\displaystyle}

\newcommand{\bmat}[1]{\begin{bmatrix}
#1
\end{bmatrix}}

\renewcommand{\Re}{\mbox{Re}}
\renewcommand{\Im}{\mbox{Im}}
\newcommand{\eps}{\varepsilon}

\numberwithin{equation}{section}
\numberwithin{figure}{section}

\begin{document}
\title{Analytic shape recovery of an elastic inclusion from elastic moment tensors\thanks{\footnotesize
This study was supported by the National Research Foundation of Korea (NRF) grant funded by the Korean government (MSIT) (NRF-2021R1A2C1011804).}}

\author{
Daehee Cho\thanks{Department of Mathematical Sciences, Korea Advanced Institute of Science and Technology, 291 Daehak-ro, Yuseong-gu, Daejeon 34141, Republic of Korea (dhcho2440@kaist.ac.kr, mklim@kaist.ac.kr).}\and
Mikyoung Lim\footnotemark[2]}

\date{\today}
\maketitle

\begin{abstract}
In this paper, we present an analytic non-iterative approach for recovering a planar isotropic elastic inclusion embedded in an unbounded medium from the elastic moment tensors (EMTs), which are coefficients for the multipole expansion of a field perturbation caused by the inclusion. EMTs contain information about the inclusion's material and geometric properties and, as is well known, the inclusion can be approximated by a disk from leading-order EMTs.  We define the complex contracted EMTs as the linear combinations of EMTs where the expansion coefficients are given from complex-valued background polynomial solutions. By using the layer potential technique for the Lam\'{e} system and the theory of conformal mapping, we derive explicit asymptotic formulas in terms of the complex contracted EMTs for the shape of the inclusion, treating the inclusion as a perturbed disk. These formulas lead us to an analytic non-iterative algorithm for elastic inclusion reconstruction using EMTs. We perform numerical experiments to demonstrate the validity and limitations of our proposed method.
\end{abstract}

%
%


\section{Introduction}
We consider the inverse problem of reconstructing a planar isotropic elastic inclusion embedded in a homogeneous isotropic background medium from exterior measurements. 
Let the unbounded background medium be isotropic with the Lam\'{e} constants $\lambda$, the bulk modulus, and $\mu$, the shear modulus. Let the isotropic elastic inclusion $\Omega$ have the Lam\'{e} constants $(\tlambda,\tmu)$. Let  $\Omega$ be a bounded simply connected domain in $\RR^2$ with a $C^{1,\gamma}$ boundary for some $\gamma\in(0,1]$.
We assume that $\lambda$, $\mu$, $\tlambda$ and $\tmu$ are known and $(\lambda-\tlambda)^2+(\mu-\tmu)^2\neq 0$. 
We denote by $\mathbb{C}_0$ and $\mathbb{C}_1$ the elastic tensors in the background and $\Omega$, respectively (see \cref{subsec:lame} for details).
Let $\H=(H_1,H_2)$ be a background displacement field, which is a vector-valued solution to $ \nabla\cdot \left(\mathbb{C}_0\widehat{\n}\H\right)=0$ in $\RR^2$.
The displacement $\bu$ in the presence of $\Omega$ satisfies
\begin{align}\label{eqn:main:trans}
\begin{cases}
\ds\n\cdot\left((\mathbb{C}_0\chi_{\R^2\setminus \Omega}+{\mathbb{C}}_1\chi_\Omega)\widehat{\n}\bu\right) =\bm{0}&\mbox{in }\mathbb{R}^2,\\
\ds\bu(x)-\H(x)=O(|x|^{-1})&\mbox{as }|x|\to+\infty,
\end{cases}
\end{align}
where $\widehat{\n}\bu$ represents the strain, that is, $ \widehat{\n}\bu =\frac{1}{2}(\n\bu+(\n\bu)^T)$.
The solution $\bu$ admits the multipole expansion (see \cite{Ammari:2004:RSI} and \cite[Theorem 10.2]{Ammari:2007:PMT}) as follows:
\beq\label{multipole:def}
\bu(x)=\H(x)+\sum_{j=1}^2 \sum_{|\alpha|\geq 1}\sum_{|\beta|\geq 1}
\frac{(-1)^{|\beta|}}{\alpha!\, \beta!}\, \p^\alpha H_j({0})\, \p^\beta\G(x)\,M_{\alpha\beta}^j,\quad|x|\gg1,
\eeq
where $\alpha$, $\beta$ are multi-indices, $M^j_{\alpha\beta}=(m^j_{\alpha\beta 1},m^j_{\alpha \beta 2})$, and $\G$ is the fundamental solution corresponding to the background elastostatic problem. The constant coefficients $m^j_{\alpha\beta 1},m^j_{\alpha \beta 2}$ are called the elastic moment tensors (EMTs).  Containing information about the inclusion's material and geometric properties, EMTs have been used as building blocks for the recovery of elastic inclusions from exterior measurements \cite{Ammari:2002:CAE,Kang:2003:IEI,Lim:2011:RSI}. In this paper, we develop an analytic non-iterative method for recovering a bounded simply connected inclusion $\Omega$ from EMTs.

Iterative methods for identifying an elastic inclusion from EMTs have been investigated using optimization approaches, where objective functions are defined based on the discrepancies between the measurement data of the target and reconstructed inclusions. 
A bounded simply connected inclusion can be analytically approximated as a disk, or a ball in three dimensions, using leading-order EMTs \cite{Ammari:2002:CAE} (see also \cite{Kang:2003:IEI}); this disk can be used as an initial guess.  
With the help of higher-order EMTs, one can further refine the shape details through iterative numerical computations \cite{Lim:2011:RSI}. 
An essential aspect of this optimization approach lies in the shape derivative of the objective function, which represents the derivative of the objective function with respect to the deformation of the inclusion's boundary.
In \cite{Lagha:2016:SPI,Lim:2011:RSI} (see also \cite{Ammari:2010:RSI2}), for the contracted EMTs--the linear combinations of EMTs with coefficients of real-valued polynomial background solutions--the shape derivative was asymptotically expressed as a boundary integral; the integrand is given by using the transmission problem solutions, with various background displacement fields, corresponding to the initial guess.

In the context of the conductivity inclusion problem, optimization approaches have been developed by utilizing the generalized polarization tensors (GPTs) in place of EMTs \cite{Ammari:2012:MIE,Ammari:2012:GPT1}. We refer readers to  \cite{Ammari:2010:CIP1} for the GPTs' shape derivative, which is in the form of boundary integral involving solutions to the transmission problem with conductivity inclusions. Furthermore, in two dimensions, explicit series expressions for the shape derivative of GPTs have been derived by considering the inclusion as a perturbed disk or perturbed ellipse  \cite{Ammari:2010:CIP1,Choi:2021:ASR}. Notably, these explicit formulas do not involve transmission problem solutions and enable the implementation of non-iterative conductivity inclusion recovery. But when it comes to the elastostatic problem, deriving an explicit expression for the shape derivative of EMTs becomes much more complicated due to the vector-valued nature of elastic problems. 
To our knowledge, explicit formulas for the shape derivative of EMTs have not yet been reported in the literature.

A powerful approach for solving plane elastostatic inclusion problems has been to apply the theory of conformal mapping, which transforms the original vector-valued problems into complex scalar-valued problems. For instance, the complex formulation for the elastostatic solution was derived by  Muskhelishvili \cite{Muskhelishvili:1953:SBP} (see also the book by Ammari--Kang \cite{Ammari:2004:RSI}). Movchan--Serkov obtained asymptotic formulas for the effective parameters of dilute composites based on complex analysis theory \cite{Movchan:1997:PSM}. Furthermore, Kang et al. studied the spectral structure for the layer potential operator of an elastostatic inclusion problem \cite{Ando:2018:SPN}. More recently, Mattei--Lim introduced a series solution expansion for the elastic system with a rigid boundary condition \cite{Mattei:2021:EAS}; we also refer readers to \cite{Cherkaev:2022:GSE,Choi:2023:IPP,Choi:2021:ASR,Jung:2021:SEL} for the conductivity inclusion problems.

In the present paper, to deal with the inclusion problem \cref{eqn:main:trans}, we utilize the layer potential technique and the complex formulation for the elastostatic problem. We introduce the concept of complex contracted EMTs that are linear combinations of EMTs where the expansion coefficients are given from complex-valued background polynomial solutions. 
By considering the inclusion as a perturbation of a disk, we derive explicit expressions for the shape derivative of complex contracted EMTs.
Based on these formulas, we propose an analytic non-iterative method to recover a bounded simply connected elastic inclusion from EMTs (refer to \cref{thm:main_result,sec:numerical}).

The rest of this paper is organized as follows. 
In \cref{sec:pre}, we outline the layer potential approach for the elastostatic problem in two dimensions and introduce the concept of EMTs. \Cref{sec:complex} is devoted to the complex formulation for the elastostatic problem. We derive solutions to the transmission problem with a circular inclusion in \cref{sec:trans:disk}. 
We derive the main asymptotic results in \cref{sec:analytic:recovery}  and present our recovery method for an elastic inclusion with numerical examples in \cref{sec:numerical}. 
The paper ends with some concluding remarks in \cref{sec:conclusion}.

\section{Preliminaries}\label{sec:pre}
\subsection{Layer potential technique for the Lam\'{e} system in two dimensions}\label{subsec:lame}
Let $\Omega,\lambda,\mu,\tlambda,\tmu$ be given as in the introduction. 
The elastic tensor $\CC_0$ associated with the Lam\'{e} constants  $(\lambda,\mu)$ admits the tensor expression, employing the Einstein summation convention, that 
\begin{align*}
\CC_0
=C_{kl}^{ij}\, e_i\otimes e_j\otimes e_k\otimes e_l
\quad\mbox{with}\quad 
C_{kl}^{ij}
=\lambda\delta_{ij}\delta_{kl}+\mu(\delta_{ik}\delta_{jl}+\delta_{il}\delta_{jk}),\ i,j,k,l=1,2,
\end{align*}
$\delta_{ij}$ being the Kronecker's delta, ${e}_1=(1,0)$, ${e}_2=(0,1)$ and $\left(e_i\otimes e_j\otimes e_k\otimes e_l\right) \bm{a}=a_{kl}\, e_i\otimes e_j$ for $\bm{a}=(a_{kl})$. We also define $ \bm{b}\otimes \bm{c}= b_{ij}c_{kl}\, e_i\otimes e_j\otimes e_k\otimes e_l $ for two tensors  $\bm{b}= (b_{ij})$ and $\bm{c}= (c_{kl})$. If $\bm{a}$ is symmetric, we have
\begin{align}\label{C:symm:a}
\mathbb{C}_0\bm{a}
=\lambda \operatorname{tr}(\bm{a})\mathbb{I}_2+2\mu\bm{a},
\end{align}
`$\operatorname{tr}$' standing for the trace of a matrix and $\mathbb{I}_2$ being the identity matrix.

Consider the displacement field $\bu$ in the presence of $\Omega$ with the far-field loading $\bu_0$. 
Assuming that there are no body forces, the background solution $\bu_0$  satisfies
\beq\label{strain_stress}
{\bm{\eps}}=\widehat{\n}\bu_0,\quad \bm{\sigma}(\bu_0)=\mathbb{C}_0\bm{\eps},\quad \nabla\cdot\bm{\sigma}(\bu_0)=0,
\eeq
where ${\bm{\eps}}$ and $\bm{\sigma}$ are called the strain and stress fields, respectively. 
This implies that
$\L_{\lambda,\mu}\bu_0=0$ in $\RR^2$ with the differential operator
\begin{align*}
\L_{\lambda,\mu}\bu:=\mu\Delta\bu+(\lambda+\mu)\nabla\nabla\cdot\bu.
\end{align*}
We define the conormal derivative--the traction term--as
\begin{align}\label{def:conormal}
\frac{\partial\bu}{\partial\nu}=\left(\mathbb{C}_0 \widehat{\n}\bu\right)N=\lambda(\nabla\cdot\bu)N+\mu(\nabla\bu+\nabla \bu^T)N\quad\mbox{on }\p \Omega,
\end{align}
where $N=(N_1,N_2)$ stands for the outward unit normal to $\p \Omega$.

We similarly denote by ${\CC}_1$ and $\L_{\tlambda,\tmu}$ the elastic tensor and differential operator associated with the Lam\'{e} constants $(\tlambda,\tmu)$, respectively.

We assume $\mu>0$, $\lambda+\mu>0$ and $\tmu>0$, $\tlambda+\tmu>0$ so that $\L_{\lambda,\mu}$ and $\L_{\tlambda,\tmu}$ are elliptic. We further assume $(\lambda-\tlambda)(\mu-\tmu)\geq0$ and $0<\tlambda,\tmu<\infty$
for the solvability of \cref{eqn:main:trans} (see \cite{Escauriaza:1993:RPS}).

Let $\G=(\Gamma_{ij})_{i,j=1}^2$ be the Kelvin matrix of the fundamental solution to the Lam\'e system $\L_{\lambda,\mu}$, that is,
\begin{gather}\notag
\Gamma_{ij}({x})
=\ds\frac{\a}{2\pi}\delta_{ij}\log|x|-\frac{\b}{2\pi}\frac{x_ix_j}{|x|^2},\quad x\neq 0,\\
\label{def:alpha:beta}
\a=\frac{1}{2}\Big(\frac{1}{\mu}+\frac{1}{2\mu+\lambda}\Big),\quad \b=\frac{1}{2}\Big(\frac{1}{\mu}-\frac{1}{2\mu+\lambda}\Big).
\end{gather}
We also set $\tilde{\a}$ and $\tilde{\b}$ as in \cref{def:alpha:beta} with $\tilde{\l}$, $\tilde{\m}$ in the place of $\lambda$, $\mu$.

The single- and double-layer potentials for the operator $\L_{\lambda,\mu}$ are defined by
 \begin{align*}
\Ss_{\p \Omega}[\bvp](x):=&\int_{\partial \Omega}\G(x-y)\bvp(y)d\sigma(y),\quad x\in\RR^2,\\
\Dd_{\p \Omega}[\bvp](x):=&\int_{\partial \Omega}\frac{\partial}{\partial \nu_y}\G(x-y)\bvp(y)d\sigma(y),\quad x\in\RR^2\setminus\p \Omega
\end{align*}
for $\bvp=(\varphi_1,\varphi_2)\in L^2(\p \Omega)^2$, ${\partial}/{\partial \nu}$ being the conormal derivative. For $i=1,2$, we have
\beq\notag
\begin{aligned}
\left(\Dd_{\p \Omega}[\bvp](x)\right)_i
&=\int_{\partial \Omega} \lambda\frac{\partial \Gamma_{ij}}{\partial y_j}(x-y)\bvp(y)\cdot N(y)
\\&\qquad
+\mu\left(\frac{\partial \Gamma_{ij}}{\partial y_k}+\frac{\partial \Gamma_{ik}}{\partial y_j}\right)(x-y) N_j(y)\varphi_k(y) d\sigma(y).
\end{aligned}
\eeq
The following jump relations hold \cite{Dahlberg:1988:BVP}:
\begin{align*}
\Dd_{\p \Omega}[\bvp]\big|^{\pm}=&\Big(\mp\frac{1}{2}I+\Kk_{\p\Omega}\Big)[\bvp]\quad\mbox{a.e. on }\partial \Omega,
\\
\frac{\partial}{\partial\nu}\Ss_{\p \Omega}[\bvp]\Big|^{\pm}=&\Big(\pm\frac{1}{2}I+\Kk^*_{\p\Omega}\Big)[\bvp]\quad\mbox{a.e. on }\partial \Omega,
\end{align*}
where the so-called Neumann--Poincar\'{e} operators $\Kk_{\p \Omega}$ and $\Kk_{\p \Omega}^*$ are defined by
\begin{align*}
\Kk_{\p \Omega}[\bvp](x) &= \mbox{p.v.}\int_{\partial \Omega}\frac{\partial}{\partial \nu_y}\G(x-y)\bvp(y)\,d\sigma(y),\\
\Kk_{\p \Omega}^*[\bvp](x)&=\mbox{p.v.} \int_{\partial \Omega} \frac{\partial}{\partial\nu_{x}}\G(x-y)\bvp(y)\,d\sigma(y).
\end{align*}
The superscripts $+$ and $-$ indicate the limits from the outside and inside of $\Omega$, respectively, 
and $\mbox{p.v.}$ means the Cauchy principle value.

We denote by $\Psi$ the space of rigid displacement, that is,
\begin{align*}
\Psi&=\left\{\bpsi=(\psi_1,\psi_2): \partial_i\psi_j+\partial_j\psi_i=0,\ 1\leq i,j\leq 2\right\}\\
&=\mbox{span}\left\{(1,0),(0,1),(x_2,-x_1)\right\}
\end{align*}
and set
$
L^2_\Psi(\partial \Omega)=\left\{\bm{f}\in L^2(\partial \Omega)^2:\int_{\partial \Omega}{\bm{f}}\cdot\bpsi d\sigma=0\ \mbox{for all }\bpsi \in \Psi\right\}.
$

Now, we consider the solution $\bu$ to \cref{eqn:main:trans}. Let $\tilde{\Ss}_{\p \Omega}$ be the single-layer potential on $\p \Omega $ corresponding to the Lam\'e constants $(\li,\mi)$.
It is well known (see \cite{Ammari:2004:RSI,Escauriaza:1993:RPS}) that
\begin{align}\label{eqn:sol:expan}
\bu(x)=\begin{cases}
\ds \H(x)+\Ss_{\p \Omega}[{\bvp}](x), &x\in\mathbb{R}^2\setminus\overline{\Omega},\\
\ds\tilde{\Ss}_{\p \Omega}[{{\bpsi}}](x), &x\in \Omega
\end{cases}
\end{align} 
with the unique pair $ (\bpsi,{\bvp}) \in L^2(\p \Omega)^2\times L^2_{\Psi}(\p \Omega)$ satisfying
\begin{align}\label{eqn:density:phi:psi}
\begin{aligned}
\begin{cases}
\ds\tilde{\Ss}_{\p \Omega}[\bpsi]\big|^--\Ss_{\p \Omega}[{\bvp}]\big|^+=\H|_{\p \Omega}\quad\mbox{on }\p\Omega,\\[1mm]
\ds\frac{\p}{\p\tilde{\nu}}\tilde{\Ss}_{\p \Omega}[\bpsi]\Big|^--\frac{\p}{\p\nu}\Ss_{\p \Omega}[{\bvp}]\Big |^+=\frac{\p \H}{\p\nu}\Big |_{\p \Omega}\quad \mbox{on }\p\Omega.
\end{cases}
\end{aligned}
\end{align}

Let $\alpha=(\alpha_1,\alpha_2),\beta=(\beta_1,\beta_2)\in \mathbb{N}^2$ be multi-indices. We set $x^{\alpha}=x_1^{\alpha_1}x_2^{\alpha_2}$. 
We can expand the background solution $\H$ and the fundamental solution $\G$ as
\begin{gather}\label{HG:taylor}
\H(x)=\sum_{j=1}^2\sum_{\alpha\in\NN^2}\frac{1}{\alpha !}\,\p^\alpha H_j({0})\,x^\alpha e_j,\\ \label{HG:taylor2}
\G(x-y)=\sum_{\beta\in \NN^2}\frac{(-1)^{|\beta|}}{\beta!}\,\p^\beta\G(x)\,y^\beta,
\end{gather}
where $y$ is supported in a compact set and $|x|$ is sufficiently large.

\begin{definition}[\cite{Ammari:2002:CAE}]
We define the elastic moment tensors (EMTs) associated with $\Omega$, $(\lambda,\mu)$ and $(\tlambda,\tmu)$ as
\begin{align}\label{def:EMTs}
M^j_{\alpha\beta}=\left(m^j_{\alpha\beta 1},m^j_{\alpha \beta 2}\right)=\int_{\partial \Omega}y^\beta \g^j_\alpha (y)d\sigma(y),
\end{align}
where $(\f^j_\alpha, \g^j_\alpha)\in L^2({\partial \Omega})^2\times L^2(\p \Omega)^2$ is the solution to
\begin{align}\label{trans:prob:basis}
\begin{aligned}
\begin{cases}
\ds\tilde{\Ss}_{\p \Omega}[\f^j_\alpha]\big|^--\Ss_{\p \Omega}[\g^j_\alpha]\big|^+=x^\alpha e_j\big|_{\partial \Omega},\\[1mm]
\ds\frac{\partial}{\partial\tilde{\nu}}\tilde{\Ss}_{\p \Omega}[\f^j_\alpha]\Big |^--\frac{\partial}{\partial \nu}\Ss_{\p \Omega}[\g^j_\alpha]\Big |^+=\frac{\partial(x^\alpha e_j)}{\partial \nu}\Big |_{\partial \Omega}.
\end{cases}
\end{aligned}
\end{align}
\end{definition}

In view of \cref{eqn:sol:expan,HG:taylor,HG:taylor2,def:EMTs}, one can easily derive \cref{multipole:def}.

\subsection{Shape derivative of the contracted EMTs: an asymptotic integral formula}

Let $\Omega$ be a $\eps$-perturbation of a domain $\Omega^{(0)}$ given by
\begin{align*}
\p \Omega=\left\{x+\eps h(x) N(x):\,x\in \p \Omega^{(0)}\right\},
\end{align*}
where $\Om^{(0)}$ is a bounded simply connected domain with $C^{2,\gamma}$ boundary, $h$ is a real-valued $C^{1,\gamma}$ function on $\p \Omega^{(0)}$ for some $\gamma\in(0,1]$, $\eps$ is a small positive parameter, and $N$ is the outward unit normal to $\p \Omega^{(0)}$. We assume that $\Omega^{(0)}$ has the same Lam\'e constants $(\tilde{\l},\tilde{\m})$ as $\Omega$.

Consider two vector-valued polynomials $\H$ and $\F$ satisfying the background elastostatic equation associated with $(\lambda,\mu)$ in $\RR^2$. They admit the Taylor expansions: for some $a_j^\alpha,b_k^\beta\in\RR$, 
\beq\label{HF:poly:expan}
\begin{aligned}
\H(x)&=\sum_{j=1,2}\sum_{\alpha\in\NN^2}a_j^\alpha x^\alpha e_j,\quad
\F(x)=\sum_{k=1,2}\sum_{\beta\in\NN^2} b_k^\beta x^\beta e_k.
\end{aligned}
\eeq

Let $\bu_\eps$ and $\bv_\eps$ be the solutions to \cref{eqn:main:trans} where the background displacement fields are given by $\H$ and $\F$, respectively. We include the subscript $\eps$ to emphasize the dependence of the solutions on the shape perturbation. 
We denote by $\bu_\eps^e$ (resp. $\bv_\eps^e$) and $\bu_\eps^i$ (resp. $\bv_\eps^i$) the restrictions of $\bu_\eps$ (resp. $\bv_\eps$) to the exterior and interior of $\Omega$. 
We also let $\bu_0$ and $\bv_0$ be defined for $\Omega^{(0)}$ instead of $\Omega$. 

By \cref{eqn:sol:expan}, we have
\begin{align}\label{u_0:exprssion}
\bu_0(x)=\begin{cases}
\ds \H(x)+\Ss_{\p \Omega^{(0)} }[{\bvp}_0](x), &x\in\mathbb{R}^2\setminus\overline{\Omega^{(0)}},\\
\ds\tilde{\Ss}_{\p \Omega^{(0)} }[{{\bpsi}_0}](x), &x\in \Omega^{(0)}
\end{cases}
\end{align} 
and
\begin{align}\label{u_eps:exprssion}
\bu_\eps(x)=\begin{cases}
\ds \H(x)+\Ss_{\p \Omega}[{\bvp}_\eps](x), &x\in\mathbb{R}^2\setminus\overline{\Omega},\\
\ds\tilde{\Ss}_{\p \Omega}[{{\bpsi}_\eps}](x), &x\in \Omega
\end{cases}
\end{align} 
with the solutions $(\bpsi_0,{\bvp}_0)$ and $ (\bpsi_\eps,{\bvp}_\eps)$ to \cref{eqn:density:phi:psi} corresponding to $\Omega^{(0)}$ and $\Omega$, respectively. 
Let $\Phi_\eps$ be the transformation from $\p \Om^{(0)}$ onto $\p \Om$ given by $\Phi_\eps(x)=x+\eps h(x)N(x)$, $x\in\p\Om^{(0)}$. By Lemma 3.4 in \cite{Lim:2011:RSI}  (see also \cite[Lemma 5.3]{Lagha:2016:SPI}), it holds that
\beq\label{densities:reglarity}
\left\|\bpsi_\eps\circ\Phi_\eps-\bpsi_0\right\|_{L^2(\p\Om^{(0)})}+\left\|\bvp_\eps\circ\Phi_\eps-\bvp_0\right\|_{L^2(\p\Om^{(0)})}=O(\eps).
\eeq
Hence, there exists a constant $C$ independent of $\eps$ such that
\beq\label{uniform:bound}
\|\bpsi_\eps\|_{L^2(\p\Om)},\ \|\bvp_\eps\|_{L^2(\p\Om)}\leq C.
\eeq

We denote by $B_R$ the ball centered at the origin with the radius $R$. 
Let $R$ be sufficiently large so that $B_R$ contains $\Om^{(0)}$ and $\Om$. 
Since $\Om$ is a $C^{1,\gamma}$ domain, from the regularity result by Li and Nirenberg \cite[Theorem 1.1]{Li:2003:EES}, $\bu_\eps\in C^{1,\gamma'}(\overline{\Om})\cap C^{1,\gamma'}(B_R\setminus \Om)$ for $\gamma'\in(0,\frac{\gamma}{2(\gamma+1)}]$ and there is a constant $M$ depending on $\lambda,\mu,\tlambda,\tmu,\gamma$ and $\Om^{(0)}$ such that 
\beq\notag
\left\|\bu_\eps\right\|_{C^{1,\gamma'}(\overline{\Om})}
+\left\|\bu_\eps\right\|_{C^{1,\gamma'}(B_R\setminus\Om)}\leq M \|\bu_\eps\|_{L^2(B_{R+1})}.
\eeq
The right-hand side of this inequality is uniformly bounded independent of $\eps$ by \cref{u_eps:exprssion,uniform:bound}. Hence, we get
\beq\label{regula:uniform1}
\left\|\bu_\eps\right\|_{C^{1,\gamma'}(\overline{\Om})}
+\left\|\bu_\eps\right\|_{C^{1,\gamma'}(B_R\setminus\Om)}\leq C.
\eeq
In \cite[Lemma 4.1]{Lim:2011:RSI}, it was shown by using \cref{u_0:exprssion,u_eps:exprssion,densities:reglarity,regula:uniform1} that
\beq
\label{regula:uniform2}
\|\nabla (\bu_\eps^e-\bu_0^e)\|_{L^\infty (\p \Om\setminus \Om^{(0)})}
+\|\nabla (\bu_\eps^i-\bu_0^i)\|_{L^\infty (\p \Om\cap\overline{ \Om^{(0)}})}
\leq C\eps^{\frac{\gamma'}{\gamma'+2}}
\eeq
for some constant $C$ independent of $\eps$. 
By the same derivation, it holds that
\beq\label{regular:uniform3}
\|\nabla (\bu_0^e-\bu_\eps^e)\|_{L^\infty (\p \Om^{(0)}\setminus \Om)}
+\|\nabla (\bu_0^i-\bu_\eps^i)\|_{L^\infty (\p \Om^{(0)}\cap\overline{ \Om})}
\leq C\eps^{\frac{\gamma'}{\gamma'+2}}.
\eeq

We consider the contracted EMTs, corresponding to $\Omega$ and $\Omega^{(0)}$, with the expansion coefficients $a^\alpha_j b^\beta_k$. The following relation was shown in the proof of \cite[Theorem 3.1]{Lim:2011:RSI}. For the sake of readers' understanding, we provide the proof.
\begin{lemma}\label{lemma:asym:2Dint}
We have
\beq\notag
\begin{aligned}
&\sum_{\alpha\beta j k} a^\alpha_j b^\beta_k m^j_{\alpha \beta k}(\Omega)-\sum_{\alpha\beta j k} a^\alpha_j b^\beta_k m^j_{\alpha \beta k}(\Omega^{(0)})\\
=&\int_{\Omega\setminus \Omega^{(0)}}({\CC}_1-\CC_0) \widehat{\n} \bu_\eps^i :\widehat{\n} \bv_0^e\, dx
-\int_{\Omega^{(0)}\setminus \Omega}({\CC}_1-\CC_0) \widehat{\n} \bu_\eps^e :\widehat{\n} \bv_0^i\, dx.
\end{aligned}
\eeq
\end{lemma}

\begin{proof}
We choose a disk $B_R$ centered at the orign with the radius $R$ sufficiently large that $B_R$ contains $\Omega$ and $\Omega^{(0)}$. 
By Green's formula, we have
\begin{align}\label{eq:Green1}
&\int_{\partial \Omega} \Big(\mathbf{F} \cdot \dfrac{\partial \mathbf{H}}{\partial \nu} - \dfrac{\partial \mathbf{F}}{\partial \nu}\cdot \mathbf{H}\Big)d\sigma = 0
\end{align}
and, by applying the formula on $B_{R}\setminus\overline{\Omega}$,
\begin{equation}\label{eq:Green2}
\begin{aligned}
 \int_{\partial \Omega} \Big(\mathbf{F} \cdot \dfrac{\partial \mathbf{u}_\varepsilon }{\partial \nu} - \dfrac{\partial \mathbf{F}}{\partial \nu}\cdot \mathbf{u}_\varepsilon \Big)d\sigma 
&= \int_{\partial B_{R}} \Big( \mathbf{F} \cdot \dfrac{\partial \mathbf{u}_\varepsilon }{\partial \nu} - \dfrac{\partial \mathbf{F}}{\partial \nu}\cdot \mathbf{u}_\varepsilon \Big)d\sigma.
\end{aligned}
\end{equation}
Similarly, it holds that
\begin{equation*}
\begin{aligned}
&\int_{\partial B_R} \Big( (\mathbf{v}_0-\mathbf{F})\cdot \dfrac{\partial (\mathbf{ u}_\varepsilon - \mathbf{u}_0 ) }{\partial \nu}  - \dfrac{\partial (\mathbf{v}_0-\mathbf{F})}{\partial \nu}\cdot (\mathbf{u}_\varepsilon -\mathbf{u}_0)\Big)d\sigma \\
=&\int_{\partial B_{R'}} \Big( (\mathbf{v}_0-\mathbf{F})\cdot \dfrac{\partial (\mathbf{ u}_\varepsilon - \mathbf{u}_0 ) }{\partial \nu}  - \dfrac{\partial (\mathbf{v}_0-\mathbf{F})}{\partial \nu}\cdot (\mathbf{u}_\varepsilon -\mathbf{u}_0)\Big) d\sigma \quad \mbox{for any }R'\geq R.
\end{aligned}
\end{equation*}
Note that $\mathcal{S}_{\partial \Omega}[\bm{f}], \nabla \mathcal{S}_{\partial \Omega}[\bm{f}] =O(|x|^{-1})$ as $|x|\rightarrow\infty$ for $\bm{f}\in L^2_{\Psi}$. By \cref{u_0:exprssion,u_eps:exprssion}, it holds that $(\mathbf{v}_0-\mathbf{F})$, $\nabla(\mathbf{v}_0-\mathbf{F})$, $(\mathbf{u}_\varepsilon-\mathbf{u}_0)$ and $\nabla(\mathbf{u}_\varepsilon-\mathbf{u}_0)$ decay as $|x|^{-1}$. Thus, by taking $R'\rightarrow\infty$, it holds that
\begin{align}\label{eq:BR4}
\int_{\partial B_R} \Big( (\mathbf{v}_0-\mathbf{F})\cdot \dfrac{\partial (\mathbf{ u}_\varepsilon - \mathbf{u}_0 ) }{\partial \nu}  - \dfrac{\partial (\mathbf{v}_0-\mathbf{F})}{\partial \nu}\cdot (\mathbf{u}_\varepsilon -\mathbf{u}_0)\Big)d\sigma =0.
\end{align}
Also, Green's formula and the jump condition for the solution to \cref{eqn:main:trans} imply that
\begin{align} \label{eq:BR5}
\int_{\partial B_R}\Big( \mathbf{v}_0\cdot\dfrac{\partial \mathbf{u}_0}{\partial \nu} - \dfrac{\partial \mathbf{v}_0}{\partial \nu} \cdot \mathbf{u}_0\Big)d\sigma
=0.
\end{align}

By the definition of the EMTs, we derive
\begin{align*}
 \sum_{\alpha\beta j k } a^{\alpha}_j b^{\beta}_k m^{j}_{\alpha\beta k}(\Omega)
 =&\sum_{\alpha\beta j k } a^{\alpha}_j b^{\beta}_k\big( \int_{\partial \Omega} y^{\beta} \mathbf{g}_{\alpha}^{j}(y)\big)\cdot e_k\,d\sigma(y)\\
 =& \int_{\partial \Omega} \mathbf{F} \cdot{\bm{\varphi}}_\eps\,d\sigma\\
 =&\int_{\partial \Omega} \mathbf{F} \cdot \Big(  \dfrac{\partial \mathcal{S}_{\partial \Omega} [{\bm{\varphi}}_\eps]}{\partial \nu}\Big|^+ - \dfrac{\partial \mathcal{S}_{\partial \Omega} [{\bm{\varphi}}_\eps]}{\partial \nu}\Big|^-\Big) \, d\sigma\\
 =& \int_{\partial \Omega}\Big( \mathbf{F} \cdot \dfrac{\partial \mathcal{S}_{\partial \Omega} [{\bm{\varphi}}_\eps]}{\partial \nu}\Big|^+ - \dfrac{\partial \mathbf{F}}{\partial \nu}\cdot \mathcal{S}_{\partial \Omega}[{\bm{\varphi}}_\eps]\Big) d\sigma,
 \end{align*}
 where $\mathbf{g}_\alpha^{j}$ is the solution to \cref{trans:prob:basis} and
the last equality follows by Green's formula for $\mathbf{F}$ and $\mathcal{S}_{\partial\Omega}[{\bm{\varphi}}_\eps]$ on $\Omega$. 
 Thus, by \cref{u_eps:exprssion,eq:Green1,eq:Green2}, we obtain
 \begin{align}\notag
  \sum_{\alpha\beta j k } a^{\alpha}_j b^{\beta}_k m^{j}_{\alpha\beta k}(\Omega)
  =& \int_{\partial \Omega} \Big(\mathbf{F} \cdot \dfrac{\partial \left(\mathbf{H}+\mathcal{S}_{\partial \Omega} [{\bm{\varphi}}_\eps] \right)}{\partial \nu}\Big|^+ - \dfrac{\partial \mathbf{F}}{\partial \nu}\cdot \left(\mathbf{H}+ \mathcal{S}_{\partial \Omega}[{\bm{\varphi}}_\eps]\right)\Big)d\sigma\\
   \label{EMT:bint:Om}
=& \int_{\partial B_R} \Big( \mathbf{F} \cdot \dfrac{\partial \mathbf{u}_\varepsilon }{\partial \nu} - \dfrac{\partial \mathbf{F}}{\partial \nu}\cdot \mathbf{u}_\varepsilon \Big)d\sigma.
\end{align}
The same relation holds for $\Omega^{(0)}$ with $\mathbf{u}_\varepsilon$ replaced by $\mathbf{u}_0$. 
By subtracting the relations \cref{EMT:bint:Om} for $\Omega$ and $\Omega^{(0)}$ and applying \cref{eq:BR4,eq:BR5}, it follows that
\begin{align}\notag
&\sum_{\alpha\beta j k } a^{\alpha}_j b^{\beta}_k m^{j}_{\alpha\beta k}(\Omega) - \sum_{\alpha\beta j k } a^{\alpha}_j b^{\beta}_k m^{j}_{\alpha\beta k}(\Omega^{(0)})\\\notag
=&\int_{\partial B_R} \Big(\mathbf{F}\cdot \dfrac{\partial (\mathbf{u}_{\varepsilon} - \mathbf{u}_0)}{\partial \nu} - \dfrac{\partial \mathbf{F}}{\partial \nu}\cdot (\mathbf{u}_\varepsilon - \mathbf{u}_0)\Big)d\sigma\\ \label{lemma:eqn:main}
=&\int_{\partial B_R} \Big(\mathbf{v}_0\cdot \dfrac{\partial \mathbf{u}_{\varepsilon}}{\partial \nu} - \dfrac{\partial \mathbf{v}_0}{\partial \nu} \cdot\mathbf{u}_\varepsilon \Big)d\sigma.
\end{align}

From the first equation in \cref{eqn:main:trans} and the divergence theorem, we obtain
$$
\begin{gathered}
\int_{B_R}\big(\mathbb{C}_0\chi_{\mathbb{R}^2\setminus \Omega}+\mathbb{C}_1\chi_{\Omega}\big)\widehat{\nabla}\mathbf{u}_{\varepsilon}:\widehat{\nabla}\mathbf{v}_0 \,dx
=
\int_{\partial B_R} \dfrac{\partial \mathbf{u}_{\varepsilon}}{\partial \nu} \cdot \mathbf{v}_0\, d\sigma,\\
\int_{B_R}\big(\mathbb{C}_0\chi_{\mathbb{R}^2\setminus \Omega^{(0)}}+\mathbb{C}_1\chi_{\Omega^{(0)}}\big)\widehat{\nabla}\mathbf{v}_{0}:\widehat{\nabla}\mathbf{u}_\varepsilon \,dx
=
\int_{\partial B_R} \dfrac{\partial \mathbf{v}_{0}}{\partial \nu} \cdot \mathbf{u}_\varepsilon\, d\sigma.
\end{gathered}
$$
By subtracting these two relations and applying \cref{lemma:eqn:main}, one can complete the proof. 
\end{proof}

We are now ready to derive the following asymptotic result. 

\begin{theorem}\label{thm:asymp:interior}
Let $\H$ and $\F$ be background solutions whose components are real-valued polynomials given by \cref{HF:poly:expan}.
Let $\bu_0$ and $\bv_0$ (resp. $\bu_\eps$ and $\bv_\eps$) be the solutions to \cref{eqn:main:trans} with $\Omega^{(0)}$ (resp. $\Om$) in the place of $\Omega$, where the background displacement fields are $\H$ and $\F$, respectively. 
For some positive $\delta>0$, it holds that as $\eps\to 0$,
 \begin{align}\notag
 \begin{aligned}
&\sum_{\alpha\beta j k} a^\alpha_j b^\beta_k m^j_{\alpha \beta k}(\Omega)
-\sum_{\alpha\beta j k} a^\alpha_j b^\beta_k m^j_{\alpha \beta k}(\Omega^{(0)})
\\=& \
 \eps\int_{\p \Omega^{(0)} \cap \,\p (\Omega\setminus \Omega^{(0)}) }  h({\CC}_1-\CC_0) \widehat{\n} \bu_0^i :\widehat{\n} \bv_0^e \,d\sigma(x)\\
&+\eps\int_{\p \Omega^{(0)}\cap\, \p (\Omega^{(0)}\setminus \Omega)} h({\CC}_1-\CC_0) \widehat{\n} \bu_0^e :\widehat{\n} \bv_0^i \,d\sigma(x)+O(\eps^{1+\delta}),
\end{aligned}
\end{align}
where $A:B=\sum_{ij}A_{ij}B_{ij}$ for two matrices $A,B$. 
\end{theorem}
\begin{proof}
 The proof follows some lines in the proof of Theorem 2.1 in \cite{Ammari:2010:RSI2}. 
 
Set $x_t=x_0+th(x_0)N(x_0)$ for $x_0\in\p\Om^{(0)}$ and $t\in[0,\eps]$. We get
\beq\label{thm:main:eqn:1}
\begin{aligned}
&\int_{\Omega\setminus \Omega^{(0)}}({\CC}_1-\CC_0) \widehat{\n} \bu_\eps^i :\widehat{\n} \bv_0^e\, dx_0
\\
=&\ \int_0^\eps\int_{\p \Om^{(0)}\cap\{h>0\}} \left(({\CC}_1-\CC_0) \widehat{\n}\bu_\eps^i(x_t) :\widehat{\n} \bv_0^e(x_t)\right)|h(x_0)|\,d\sigma(x_0)\,dt+O(\eps^2) 
\end{aligned}
\eeq
and
\beq\label{thm:main:eqn:2}
\begin{aligned}
&\int_{\Omega^{(0)}\setminus \Omega}({\CC}_1-\CC_0) \widehat{\n} \bu_\eps^e :\widehat{\n} \bv_0^i\, dx_0
\\
=&\ \int_0^\eps\int_{\p \Om^{(0)}\cap\{h<0\}} \left(({\CC}_1-\CC_0) \widehat{\n}\bu_\eps^e(x_t) :\widehat{\n} \bv_0^i(x_t)\right)|h(x_0)|\,d\sigma(x_0)\,dt+O(\eps^2).
\end{aligned}
\eeq
By using the estimates \cref{regula:uniform1,regular:uniform3} for $\bu_\eps$ and $\bv_0$, we derive that, for $x_0\in\p \Om^{(0)}\cap\{h>0\}$, 
\begin{align}\notag
({\CC}_1-\CC_0) \widehat{\n}\bu_\eps^i(x_t) :\widehat{\n} \bv_0^e(x_t)
&=({\CC}_1-\CC_0) \widehat{\n}\bu_\eps^i(x_0) :\widehat{\n} \bv_0^e(x_0)+O(\eps^{\gamma'}) \\ \label{tensor:asymp:1}
&=({\CC}_1-\CC_0) \widehat{\n}\bu_0^i(x_0) :\widehat{\n} \bv_0^e(x_0)+O(\eps^{\frac{\gamma'}{\gamma'+2}}).
\end{align}
Similarly, for $x_0\in \p \Om^{(0)}\cap \{h<0\}$, we get
\begin{align}\label{tensor:asymp:2}
({\CC}_1-\CC_0) \widehat{\n}\bu_\eps^e(x_t) :\widehat{\n} \bv_0^i(x_t)
&=({\CC}_1-\CC_0) \widehat{\n}\bu_0^e(x_0) :\widehat{\n} \bv_0^i(x_0)+O(\eps^{\frac{\gamma'}{\gamma'+2}}).
\end{align}
By \cref{lemma:asym:2Dint,tensor:asymp:1,tensor:asymp:2,thm:main:eqn:1,thm:main:eqn:2}, we prove the theorem. 
\end{proof}

\begin{remark}
Although \cref{thm:asymp:interior} is derived for a single inclusion, it can be extended to multiple separated inclusions. Indeed, the proof of \cref{thm:asymp:interior} mainly relies on the layer potential technique and the regularity result by Li and Nirenberg \cite[Theorem 1.1]{Li:2003:EES}, both of which are applicable to multiple inclusions. Additionally, for a conductivity inclusion problem, an iterative level-set based optimization method has been developed to recover multiple inclusions \cite{Ammari:2013:NOC}. We believe this iterative method can be generalized to elastic inclusion problems. However, the analytic shape recovery formulas presented in \cref{subsec:explicit:shape} are specifically for a perturbed disk and, therefore, assume a single inclusion. 
It is not straightforward to generalize these formulas to multiple inclusions since, in this case, the EMTs simultaneously depend on multiple shape deformation functions.  Nonetheless, exploring such a generalization remains an intriguing direction for future research.

\end{remark}

\section{Complex-variable formulation}\label{sec:complex}
We identify the coordinate vector $x=(x_1,x_2)\in\RR^2$ with the complex variable $z=x_1+\rmi x_2\in\CC$.
Let $\Re\{\cdot\}$ and $\Im\{\cdot\}$ denote the real and imaginary parts of a complex number, respectively, and $(\cdot)_j$ the $j$-th component of a vector in $\RR^2$. 
We  rewrite a vector-valued function $\bu(x)$ as the complex function 
$$u(z):=\big((\bu)_1+\rmi (\bu)_2\big)(x)$$
and, in the same way, define the complex-valued functions $\vp$, $\psi$, $H$ and $F$ corresponding to $\bvp$, $\bpsi$, $\H$ and $\F$, respectively. We also define the complex-valued single-layer potentials for complex-valued density functions $\varphi$, $\psi$ as
\begin{align*}
S_{\p \Omega}[\varphi](z):=&\,\big((\Ss_{\p \Omega}[\bvp])_1+\rmi (\Ss_{\p \Omega}[\bvp])_2\big)(x),\\
{\tilde{S}}_{\p \Omega}[\psi](z):=&\,\big(({\tilde{\Ss}}_{\p \Omega}[\bpsi])_1+\rmi ({\tilde{\Ss}}_{\p \Omega}[\bpsi])_2\big)(x).
\end{align*}

\subsection{Complex contracted EMTs}
One can express the solution to the Lam\'{e} system in terms of complex holomorphic functions as follows. 
\begin{lemma}[\cite{Ammari:2007:PMT,Muskhelishvili:1953:SBP}]\label{Complex_representation_of_solutions}
Let $\Omega$ be a simply connected domain in $ \R^2 $ with Lam\'e constants $(\lambda,\mu)$.
Let $\bv\in H^1(\Omega)^2$ satisfy $\L_{\l,\m}\bv={0}$ in $\Omega$. 
Then there are holomorphic functions $f$ and $g$ in $\Omega$ such that
\begin{align}\label{Complex_representation_of_solutions_1}
\left((\bv)_1+\rmi (\bv)_2\right)(z)={\kappa} f(z)-z\overline{f'(z)}-\overline{g(z)}\quad\mbox{with }\kappa=\frac{\lambda+3\mu}{\lambda+\mu}.
\end{align}
The conormal derivative of $\bv$ satisfies
\begin{align}\label{Complex_representation_of_solutions_2}
\left[\left(\frac{\partial \bv}{\partial \nu}\right)_1+\rmi \left(\frac{\partial \bv}{\partial \nu}\right)_2\right]d\sigma(z)=-2\rmi \mu\,\partial\left[f(z)+z\overline{f'(z)}+\overline{g(z)}\right],
\end{align}
where 
$\partial=\frac{\partial}{\partial x_1}dx_1+\frac{\partial}{\partial x_2}dx_2=\frac{\p}{\p z}dz+\frac{\p}{\p\overline{z}}d\overline{z}$ and $f'(z)=\frac{\p f}{\p z}(z)$. 
\end{lemma}

Recall that $\H$ and $\F$ are solutions to $ \L_{\l,\m}\bv={0} $ in $\RR^2$. 
In view of \cref{Complex_representation_of_solutions}, one can expand the complex functions $H(z)$ and $F(z)$ into the basis functions
\begin{align}
\label{H:exp:basis:complex:center0}
\kappa pz^{n}-z\overline{pnz^{n-1}} + \overline{qz^{n}}
\end{align}
with $n\in\mathbb{N}$ and $(p,q)=(0,1),(0,\rmi),(1,0),(\rmi,0)$ using real-valued expansion coefficients, neglecting the constant term. 
In other words, we have the basis functions
\beq \label{def:H:family:center0}
\begin{cases}
\ds h_n^{(1)}(z)= \overline{z^{n}}, \\
 \ds h_n^{(2)}(z)=\overline{\,\rmi z^{n}},\\
\ds h_n^{(3)}(z)=\kappa z^{n}-z\,\overline{nz^{n-1}},\\
\ds h_n^{(4)}(z)=\kappa \rmi z^{n}-z\overline{\,\rmi nz^{n-1}}.
\end{cases}
\eeq
\begin{definition}[Complex contracted EMTs]\label{def:CCEMTs}
Let $n,m\in\mathbb{N}$ and $t,s=1,2,3,4$.
Let the complex expressions of $\mathbf{H}, \mathbf{F}$ be given by $H=h_n^{(t)}, F=h_{m}^{(s)}$. We define
\begin{align}\label{def:Enmts:C}
\mathbb{E}_{nm}^{(t,s)}(\Omega) = \sum_{\alpha\beta j k }a^{\alpha}_j b^{\beta}_k m^{j}_{\alpha \beta k}(\Omega)
\end{align}
where  $a^{\alpha}_j, b^{\beta}_k$ are the coefficients of the Taylor series expansions of $\mathbf{H}, \mathbf{F}$ as in \cref{HF:poly:expan}.
We call $\mathbb{E}_{nm}^{(t,s)}(\Omega)$ the complex contracted EMTs associated with $\Omega$. 
\end{definition}

From the definition of  $m^{j}_{\alpha \beta k}(\Omega)$ in \cref{def:EMTs}, one can easily find that
\begin{align}
\label{contracted_EMT:integral_form}
\mathbb{E}^{(t,s)}_{nm}(\Omega)=\int_{\partial \Omega} \mathbf{F}\cdot \bm{\varphi}\,d\sigma,
\end{align}
where $\mathbf{H},\mathbf{F}$ are given as in \cref{def:CCEMTs} and $\bm{\varphi}$ is the density function corresponding to $\mathbf{H}$ given by \cref{eqn:sol:expan}.

\subsection{Complex integral expression for the single-layer potential}
One can express the complex single-layer potential $S_{\p \Omega}[\varphi] (z)$ as in \cref{Complex_representation_of_solutions_1}. 
It was shown in \cite[Chapter 9.4]{Ammari:2007:PMT} and \cite{Ando:2018:SPN} that in $\Omega$ (or in $\CC\setminus\overline{\Omega}$),
\beq\label{singlelayer:fg}
\begin{aligned}
2 S_{\p \Omega}[\varphi] (z)&=\kappa f(z)-z\overline{f'(z)}-\overline{g(z)},\\
\ds	f(z)&=\b\,\L[\vp](z),\\
g(z)&=-\a\,\L[\overline{\vp}](z)-\b\,\C[\overline{\zeta}\vp](z)+c
\end{aligned}
\eeq
with the complex operators defined by
\begin{align}\notag
\L[\vp](z)&=\frac{1}{2\pi}\int_{\p \Omega}\log(z-\zeta)\vp(\zeta)\,d\sigma(\zeta),\\\notag
\C[\vp](z)&=\frac{1}{2\pi}\int_{\p \Omega}\frac{\vp(\zeta)}{z-\zeta}\,d\sigma(\zeta),\\\notag
\C[\overline{\zeta}\vp](z)&=\frac{1}{2\pi}\int_{\p \Omega}\frac{\vp(\zeta)}{z-\zeta}\,\overline{\zeta}\,d\sigma(\zeta)
\end{align}
and $c=\frac{\beta}{2\pi}{\int_{\partial \Omega}\overline{\varphi(\zeta)}d\sigma(\zeta)}$. 
 Note that $\kappa\beta=\alpha$ and $\L[\vp]'(z):=\frac{\p}{\p z}\L[\vp](z)=\C[\vp](z)$. 
 If $\bm{\varphi}\in L^2_\Psi(\p \Omega)$, then $c=0$. 
 
 We can rewrite \cref{singlelayer:fg} as
\begin{align}\label{asymptotic_general_single_layer}
2 S_{\p \Omega}[\vp](z)=
2\a\,L[\vp](z)-\b z\,\overline{\C[\vp](z)}+\b\,\overline{\C[\overline{\zeta}\vp](z)}-\overline{c}\quad\mbox{in }\CC
\end{align}
by using the single-layer potential for the Laplacian
\begin{align}
\label{def:L}
L[\vp](z)&=\frac{1}{2\pi}\int_{\p \Omega}\ln|z-\zeta|\vp(\zeta)\,d\sigma(\zeta).
\end{align}
\section{Transmission problem for a disk}\label{sec:trans:disk}

We consider the case in which the elastic inclusion is a disk, namely $D$, centered at $a_0\in \CC$ with the radius $\gamma>0$. 
Then the mapping $\Psi$ defined by
\beq\label{Psi:disk}
\Psi(w)=w+a_0,\quad |w|\geq \gamma
\eeq
is a conformal mapping from the exterior of the unit disk centered at $0$ to the exterior of $D$. 
We denote by $F_m$ the Faber polynomials associated with $\Psi$ given by \cref{Psi:disk} (see \cite{Faber:1903:UPE, Jung:2021:SEL}). It holds that $$F_m(z)=(z-a_0)^m,\quad m=0,1,2,\dots.$$
For $z=\Psi(w)\in\p D$ with $w=\gamma e^{\rmi \theta}$,
we define density basis functions
$$\varphi_k(z)=\gamma^{-1}{e^{\rmi k \theta}}\quad \mbox{for } k\in\ZZ.$$

\subsection{Single-layer potential}\label{subsec:single:expan}
The operator $\mathcal{L}$ satisfies that for $m\geq 1$,
\beq
\begin{aligned}
\mathcal{L} \left[ \varphi_{m} \right](z) 
&= \begin{cases}
-\dfrac{1}{m}\gamma^{-m}F_m(z) &\mbox{in }D,\\[2mm]
0&\mbox{in }\CC\setminus\overline{D},
\end{cases}\\
\mathcal{L} \left[ \varphi_{-m} \right](z) 
&= \begin{cases}
0 &\mbox{in }D,\\[2mm]
-\dfrac{1}{m}\gamma^{m}(z-a_0)^{-m}&\mbox{in }\CC\setminus\overline{D}.
\end{cases}
\end{aligned}
\eeq
As shown in \cite{Jung:2021:SEL}, the operator $L$ satisfies that for $m\geq1$, 
\beq\label{L:varphi_m:exp}
\begin{aligned}
L\left[ \varphi_{-m}\right](z)&=\overline{L\left[ \varphi_m\right](z)},\\
L\left[ \varphi_m\right](z)
&=\begin{cases}
\ds -\frac{1}{2m} \gamma^{-m}\,F_m(z),\quad&z\in\overline{D},\\[2mm]
\ds-\frac{1}{2m} \gamma^{m}\,\overline{w^{-m}},\quad&z\in\CC\setminus\overline{D}.
\end{cases}
\end{aligned}
\eeq
For notational convenience, we define for $m\geq0$ that
\begin{align*}
\widetilde{F}_0(z)&=0,\\
\widetilde{F}_m(z)&= \gamma^{-m}(z-a_0)^{m-1}\qquad\qquad\qquad\qquad\mbox{for }z\in D,\  m\neq0,\\
G_{-m}(z)&=\gamma^m\, (z-a_0)^{-m-1}=\gamma^{m}\, w^{-m-1}\qquad\mbox{for }z=\Psi(w)\in\CC\setminus\overline{D}. 
\end{align*}
From Lemma 3.1 and Lemma 3.2 in \cite{Mattei:2021:EAS}, it holds that for $m\geq 1$,
\begin{align}\notag
\mathcal{C}\left[\varphi_{m}\right](z)&=
\begin{cases}
\ds -\wtF_{m}(z)\quad&\mbox{in }D,\\[2mm]
\ds 0\quad&\mbox{in }\CC\setminus\overline{D},
\end{cases}\\\label{Cvarphim:plus:minus}
\mathcal{C}\left[\varphi_{-m}\right](z)&=
\begin{cases}
\ds 0\quad&\mbox{in } D,\\[2mm]
\ds G_{-m}(z)\quad&\mbox{in }\CC\setminus\overline{D}.
\end{cases}
\end{align}
Note that 
$$\mathcal{C}\left[\overline{\zeta}\,\varphi_{k}\right]=
\gamma\mathcal{C}\left[ \varphi_{k-1}\right]+\overline{a_0}\,\mathcal{C}\left[\varphi_k\right]\quad\mbox{for all }k\in\ZZ.$$

\subsubsection{Exterior expansion}
Let $\bvp\in L^2_\Psi(\p D)$ be given by
\beq\label{varphi:expan}
\vp(z)=\sum_{k\in\ZZ}c_k \varphi_k(\theta)\quad\mbox{with } c_k\in\CC,\ c_0=0.
\eeq
We have from \cref{L:varphi_m:exp,Cvarphim:plus:minus} that for $z=\Psi(w)\in\CC\setminus{D}$,
\begin{align}
2 L \left[ \varphi \right] (z)
&= -\sum \left(\dfrac{c_{-m}}{m}\gamma^{m}(z-a_0)^{-m} + \dfrac{c_m}{m}\gamma^{m}\overline{(z-a_0)^{-m}}\right),
\\ 
\label{C:phi:ext}
\mathcal{C}[\varphi](z)
&=\sum c_{-m}\,G_{-m}(z)
=\sum c_{-m+1}\,\gamma^{m-1}\, (z-a_0)^{-m},\\\notag
%
\mathcal{C}\left[\overline{\zeta}\,\varphi\right](z)
&= \gamma c_1 G_0(z)+\sum \gamma c_{-m}G_{-m-1}(z)+  \overline{a_0}\,\mathcal{C}[\varphi](z) \\\label{C:zera:ext}
&=\sum\left( c_{-m+2}\gamma^m + \overline{a_0}\, c_{-m+1}\gamma^{m-1}  \right)(z-a_0)^{-m},
\end{align}
where $\sum$ means $\sum_{m=1}^\infty$.

We express $2S_{\p D}[\vp](z)$ in terms of $z-a_0$ by using \cref{asymptotic_general_single_layer} that
\beq\label{Scal:ext}
\begin{aligned}
2S_{\p D}[\vp](z)
=&-\a\sum\left(\frac{c_{-m}}{m}\,\gamma^{m}\,(z-a_0)^{-m}+\frac{c_m}{m}\,\gamma^{m}\,\overline{(z-a_0)^{-m}}\right)\\
&-\beta z\sum\overline{c_{-m+1}}\, \gamma^{m-1}\,\overline{(z-a_0)^{-m}}\\
&+\beta\sum\left( \overline{c_{-m+2}}\,\gamma^m +{a_0}\, \overline{c_{-m+1}}\gamma^{m-1}  \right)\overline{(z-a_0)^{-m}}
\end{aligned}
\eeq
and, from the continuity of $\Scal_{\p D}$ to the boundary of $D$,
\begin{align}\label{Scal:pD:plus}
&2S_{\p D}[\vp]\Big|^+_{\p D}(z)
=
\beta\,\overline{c_1}\gamma\varphi_1-\a \sum  \frac{c_{-m}}{m}\gamma\varphi_{-m}
-\a\sum\frac{c_m}{m}\gamma\varphi_m.
\end{align}

Using $z=\Psi(w)=w+a_0\in\p D$ with $w=\gamma e^{\rmi \theta}$on $\p D$, one can easily find that
\beq\label{bound:relation}
dz=\gamma \rmi e^{\rmi \theta}d\theta=\rmi wd\theta,\quad d\overline{z}=-\gamma \rmi e^{-\rmi \theta}\,d\theta=-\rmi\overline{w}\, d\theta \quad\mbox{on }\p D.
\eeq 
By \cref{Complex_representation_of_solutions_2,singlelayer:fg}, we have
\begin{align}\notag
&\big(\left(\p_\nu\Scal_{\p D}[\bvp]\right)_1+\rmi\left(\p_\nu\Scal_{\p D}[\bvp]\right)_2\big)\Big|^+_{\p D}(z)\,d\sigma(z)\\\notag
=&-\rmi\b\mu\left[\C[\vp](z)+\overline{\C[\vp](z)}\right]dz
+\rmi\mu\left[\a\,\overline{\C[\overline{\vp}](z)}-\b\left(z\,\overline{\C[\vp]'(z)}-\overline{\C[\overline{\zeta}\vp]'(z)}\right)\right]d\overline{z}
\\ \label{Scal:dnu:ext}
=&\,\b \mu w\left[\C[\vp](z)+\overline{\C[\vp](z)}\right]d\theta
+\mu\overline{w}\left[\a\,\overline{\C[\overline{\vp}](z)}-\b\left(z\,\overline{\C[\vp]'(z)}-\overline{\C[\overline{\zeta}\vp]'(z)}\right)\right]d\theta.
\end{align}
From \cref{C:phi:int,C:zeta:phi:int}, we obtain on $\p D$ that
\begin{align*}
\overline{w}\,\overline{\C[\overline{{\vp}}](z)}
&=\sum c_m \gamma \varphi_m,\\
w\left(\C[\vp](z)+\overline{\C[\vp](z)}\right)
&=-\overline{c_1} \gamma \varphi_1 + \sum c_{-m} \gamma \varphi_{-m} + \sum \overline{c_{-m+2}} \gamma \varphi_{m},\\
\overline{w}\left(z\,\overline{\C[{\vp}]'(z)}-\overline{\C[\overline{\zeta}{\vp}]'(z)}\right)
&=\sum\overline{c_{-m+2}} \gamma \varphi_m.
\end{align*}
It follows that
\begin{align}\label{nuScal:pD:plus}
\begin{aligned}
&\big(\left(\p_\nu\Scal_{\p D}[\bm{\vp}]\right)_1+\rmi\left(\p_\nu\Scal_{\p D}[\bm{\vp}]\right)_2\big)\Big|^+_{\p D}(z)\,d\sigma(z)
\\
=&\,\mu\gamma \bigg[-\beta\, \overline{c_1}\varphi_1
 +\beta\sum c_{-m}\varphi_{-m} +\alpha\sum c_m\varphi_m \bigg]d\theta.
\end{aligned}
\end{align}

\subsubsection{Interior expansion}

Let $\bpsi\in L^2(\p D)$ given by
\beq\label{tvarphi:expan}
\psi(z)=\sum_{k\in\ZZ}d_k \varphi_k(\theta)\quad\mbox{with } d_k\in\CC.
\eeq
We have from \cref{L:varphi_m:exp,Cvarphim:plus:minus} that for $z\in D$,
\begin{align}
2L\left[ \psi \right](z) 
&= - \sum \left( \dfrac{d_m}{m}\gamma^{-m}F_m(z) + \dfrac{d_{-m}}{m}\gamma^{-m}\overline{F_m(z)} \right),
\\
\label{C:phi:int}
\mathcal{C}[\psi](z)
&=-\sum d_m  \widetilde{F}_m(z)
=-d_1\gamma^{-1}-\sum d_{m+1}\,\gamma^{-m-1}\, (z-a_0)^m,\\ \notag
\mathcal{C}\left[\overline{\zeta}\,\psi\right](z)
&=\sum d_m\left(-\gamma\widetilde{F}_{m-1}(z)-\overline{a_0}\,\widetilde{F}_m(z)\right)\\\label{C:zeta:phi:int}
&= -d_1\,\overline{a_0}\gamma^{-1}-d_2
-\sum\left(d_{m+2}\,\gamma^{-m}+d_{m+1}\,\overline{a_0}\,\gamma^{-m-1}   \right)(z-a_0)^m,
\end{align}
where $\sum$ means $\sum_{m=1}^\infty$.

One can easily derive by using \cref{asymptotic_general_single_layer,L:varphi_m:exp} that
\beq\label{Scal:int:expan}
\begin{aligned}
2\tilde{S}_{\p D}[\psi](z)
=&-\tilde{\a}\sum\Big(\frac{d_m}{m}\gamma^{-m}(z-a_0)^m+\frac{d_{-m}}{m}\gamma^{-m}\,\overline{(z-a_0)^m}\Big)\\
&+\tilde{\b}\left((z-a_0)+a_0\right)
\Big[\,\overline{d_1}\gamma^{-1}+\sum \overline{d_{m+1}}\,\gamma^{-m-1}\,\overline{(z-a_0)^m}\,\Big]\\
&-\tilde{\b} \,\overline{d_1}\,{a_0}\gamma^{-1}-\tilde{\b}\,\overline{d_2}
-\tilde{\b}\sum\left(\overline{d_{m+2}}\,\gamma^{-m}+\overline{d_{m+1}}\,{a_0}\,\gamma^{-m-1}   \right)\overline{(z-a_0)^m}
\\
&-\tilde{\b}d_0
\end{aligned}
\eeq
and, thus,
\begin{align}\label{tildeS:pD:int}
2\tilde{S}_{\p D}[\psi]\Big|^-_{\p D}(z)
=&-\tilde{\b}d_0 + \tilde{\b}\, \overline{d_1} \gamma \varphi_1
-\tilde{\a}\sum\frac{d_m}{m} \gamma \varphi_m
-\tilde{\a}\sum\frac{d_{-m}}{m} \gamma \varphi_{-m}.
\end{align}

Similar to \cref{Scal:dnu:ext}, we have
\begin{align*}
&\left(\big(\p_{\tilde{\nu}}{\tilde{\Scal}}_{\p D}[\bpsi]\big)_1+\rmi\big(\p_{\tilde{\nu}}{\tilde{\Scal}}_{\p D}[\bpsi]\big)_2\right)\Big|^-_{\p D}(z)\,d\sigma(z)\\
=&\,\tilde{\b}\tilde{\mu}w\left[
\C[\psi](z)+\overline{\C[\psi](z)}
\right]d\theta
+\tilde{\mu}\overline{w}\left[
\tilde{\a}\,\overline{\C[\overline{\psi}](z)}
-\tilde{\b}\left(
z\,\overline{\C[\psi]'(z)}-\overline{\C[\overline{\zeta}\psi]'(z)}
\right) 
\right] d\theta.
\end{align*}

From \cref{C:phi:int,C:zeta:phi:int}, we obtain on $\p D$ that
\begin{align*}
\overline{w}\,\overline{\C[\overline{\psi}](z)}
&=-\sum d_{-m} \gamma\varphi_{-m},\\
w\left(
\C[\psi](z)+\overline{\C[\psi](z)}
\right)
&=-\overline{d_2} - \overline{d_1}\,\gamma\varphi_1 -\sum \overline{d_{m+2}}\,\gamma\varphi_{-m}-\sum d_m \gamma\varphi_m ,\\
\overline{w}\left(z\,\overline{\C[\psi]'(z)}-\overline{\C[\overline{\zeta}\psi]'(z)}\right)
&=-\overline{d_2}-\sum\overline{d_{m+2}}\,\gamma\varphi_{-m}.
\end{align*}
It follows that
\begin{align}\label{conomal:int}
\begin{aligned}
&\left(\big(\p_{\tilde{\nu}}{\tilde{\Scal}}_{\p D}[\bpsi]\big)_1+\rmi\big(\p_{\tilde{\nu}}{\tilde{\Scal}}_{\p D}[\bpsi]\big)_2\right)\Big|^-_{\p D}(z)\,d\sigma(z)\\
=&-\tilde{\mu}\gamma 
\bigg[\,\tilde{\b}\,\overline{d_1}\,\varphi_1+\tilde{\a}\sum d_{-m}\varphi_{-m}  +\tilde{\b}\sum d_m \varphi_m  \,   \bigg]d\theta.
\end{aligned}
\end{align}

\subsection{Series solutions for the transmission problem}\label{subsec:trans:disk}
We find the solution $\bu$ to \cref{eqn:main:trans} for $\H$ given by \cref{H:exp:basis:complex:center0} with $p=0$, that is,
\beq\label{H:p0}
H(z)= \overline{q(z-a_0)^n}\quad\mbox{for }n\in\NN.
\eeq
We denote by $(\bpsi,\bm{\vp}) \in L^2(\p D)^2\times L^2_{\Psi}(\p D)^2$ the solution to \cref{eqn:density:phi:psi} and set $(\psi,\vp)$ to be the corresponding complex functions. 
From \cref{eqn:sol:expan}, we have
\begin{align}\notag
u(z)=\begin{cases}
\ds H(z)+S_{\p D}[\vp](z), &z\in\CC\setminus\overline{D},\\
\ds \tilde{S}_{\p D}[\psi](z), &z\in D,
\end{cases}
\end{align} 
where the transmission conditions in \cref{eqn:density:phi:psi} can be written in complex form as 
\beq\label{H:S:tS:u}
\begin{gathered}
u\big|^+_{\p D}=u\big|^-_{\p D},\\
(\partial_\nu \bm{u})_1+\rmi(\partial_\nu \bm{u})_2\big|^+_{\p D}
=(\partial_{\tilde{\nu}} \bm{u})_1+\rmi (\partial_{\tilde{\nu}} \bm{u})_2\big|^-_{\p D}
\end{gathered}
\eeq
with 
\beq\label{H:S:tS}
\begin{aligned}
(\partial_\nu \bm{u})_1+\rmi (\partial_\nu \bm{u})_2\big|^+_{\p D}&=(\partial_\nu \H)_1+\rmi (\partial_\nu \H)_2
+\left(\p_\nu\Scal_{\p D}[\bm{\vp}]\right)_1+\rmi \left(\p_\nu\Scal_{\p D}[\bm{\vp}]\right)_2\big|_{\p D}^+,\\
(\partial_{\tilde{\nu}} \bm{u})_1+\rmi (\partial_{\tilde{\nu}} \bm{u})_2\big|^-_{\p D}
&=\big(\p_{\tilde{\nu}}\tilde{\Scal}_{\p D}[\bpsi]\big)_1+\rmi \big(\p_{\tilde{\nu}}\tilde{\Scal}_{\p D}[\bpsi]\big)_2\big|^-_{\p D}.
\end{aligned}
\eeq

From \cref{Complex_representation_of_solutions_2,bound:relation}, it holds that 
\begin{gather}\label{H:pq:pD}
H=\overline{q}\,\gamma^{n+1} \varphi_{-n},\\
\left((\partial_\nu \H)_1+\rmi (\partial_\nu \H)_2\right) d\sigma(z)
= 2\mu \,\overline{q}\,n\gamma^{n+1} \varphi_{-n} d\theta\quad\mbox{on }\p D.
\end{gather}
Since $\tilde{\vp}$ and $\vp$ admit the series expansions \cref{tvarphi:expan,varphi:expan}, we can find the coefficients $d_k$ and $c_k$ in the expansions by using \cref{H:S:tS:u,H:S:tS} as follows:
\beq \label{cndn:ngeq3:mn}
\begin{gathered}
\ds\tilde{\varphi}=d_{-n}\varphi_{-n},\quad
\ds{\varphi}=c_{-n}\varphi_{-n},\\
\bmat{c_{-n}\\d_{-n}} 
= \overline{q}\,n\gamma^{n}\,\frac{2}{\tilde{\alpha}(\tilde{\mu}{\a}+\mu\b)}
\bmat{\tilde{\mu}\tilde{\a} - \mu\tilde{\a}\\ -1}.
\end{gathered}
\eeq
From \cref{Scal:ext,Scal:int:expan}, it holds that
\begin{align}\label{Scal:int:p0}
2{\tilde{S}}_{\p D}[\psi](z)
=&\, -\tilde{\a}\,\frac{d_{-n}}{n}\gamma^{-n}\,\overline{(z-a_0)^n}\quad\mbox{on }D,\\ \notag
2{S}_{\p D}[{\vp}](z)
=&\,-\alpha\frac{c_{-n}}{n}\gamma^n \,{(z-a_0)^{-n}}\\ \label{Scal:ext:p0}
&-\beta \,\overline{c_{-n}} \left((z-a_0)\gamma^n\,\overline{(z-a_0)^{-n-1}}
-\gamma^{n+2}\,\overline{(z-a_0)^{-n-2}}\right)\quad\mbox{on }\CC\setminus{D}.
\end{align}

\section{Analytic shape recovery}\label{sec:analytic:recovery}
We derive explicit formulas that reconstruct an elastic inclusion $\Omega$ from the elastic moment tensors. We regard $\Omega$ as a perturbation of a disk, where the disk can be obtained by using leading-order EMTs (\cref{subsec:approx:disk}). We then introduce the modified EMTs using the disk (\cref{subsec:modiEMTs})  and apply the asymptotic integral expression in \cref{thm:asymp:interior} (\cref{subsec:explicit:shape}).

\subsection{Approximation of the inclusion by a disk}\label{subsec:approx:disk}
We find a disk, namely $D$, that admits the same EMTs of leading-orders, more precisely $\mathbb{E}_{11}^{(1,1)}$, $\mathbb{E}_{12}^{(1,1)}$, and $\mathbb{E}_{12}^{(1,2)}$, to $\Omega$.  
We choose the center $a_0$ and the radius $\gamma$ of $D$ as follows.

Set  $$H(z)=h_1^{(1)}(z)=\overline{z}=\overline{z-a_0}+\overline{a_0}.$$ The density function $\varphi$ corresponding to the transmission problem \cref{eqn:density:phi:psi} with $H(z)=h_1^{(1)}$ is the same as $\varphi$ given by  \cref{cndn:ngeq3:mn} with $n=1$ and $q=1$. In other words,
\beq\label{def:M0}
\varphi=M_0 \gamma\varphi_{-1} \quad \mbox{with }M_0=\dfrac{2(\tilde{\mu}-\mu)}{\tilde{\mu}\alpha + \mu\beta}.
\eeq
We now take $F=h_1^{(1)}, h_2^{(1)}, h_2^{(2)}$ in \cref{def:H:family:center0}, that is, $F(z)=\overline{z}, \overline{z^2},\overline{\,\rmi z^2}$. This admits the Fourier series expansion $$F(z)=\gamma^2\varphi_{-1}+\overline{a_0}, \ \gamma^3\varphi_{-2}+2\overline{a_0}\gamma^2\varphi_{-1}+\overline{a_0}^2, \ -\rmi\left(\gamma^3\varphi_{-2}+2\overline{a_0}\gamma^2\varphi_{-1}+\overline{a_0}^2\right)\quad\mbox{on }\p D.$$
One can easily find by using \cref{contracted_EMT:integral_form} that
\begin{align}\label{EMTs:disk:D}
\begin{aligned}
\mathbb{E}^{(1,1)}_{11}(D)&=2\pi \gamma^2 M_0,
\\
\mathbb{E}^{(1,1)}_{12}(D)&= 2\pi\gamma^2 M_0 \left(a_0+\overline{a_0}\right),
\\
\mathbb{E}^{(1,2)}_{12}(D)&=2\pi\gamma^2 M_0\rmi \left(a_0-\overline{a_0}\right).
\end{aligned}
\end{align}
We arrive at the expression for $a_0$ and $\gamma$:
\beq\label{app_disk}
\begin{aligned}
\gamma&=\frac{1}{\sqrt{2\pi M_0}}\sqrt{\mathbb{E}^{(1,1)}_{11}(D)},\\
a_0&=\frac{1}{4\pi\gamma^2 M_0}\left( \mathbb{E}^{(1,1)}_{12}(D)-\rmi\, \mathbb{E}^{(1,2)}_{12}(D)\right).
\end{aligned}
\eeq

\subsection{Modified EMTs using the approximated disk $D$}\label{subsec:modiEMTs}
Let $D$ be the disk where the center $a_0$ and radius $\gamma$ are determined by \cref{app_disk}. 
As before, we set $H(z)=H_1+\rmi H_2$ for the background solution $\H$ to the problem $\L_{\l,\m}\bv=\bm{0}$. 
By replacing $z$ with $z-{a_0}$ in \cref{def:H:family:center0}, we obtain modified basis functions for $H(z)$:
\beq\label{def:H:family:a0}
\begin{cases}
\ds{\widetilde{h}}_n^{(1)}(z)= \overline{(z-a_0)^{n}}, \\[.5mm]
\ds {\widetilde{h}}_n^{(2)}(z)=\overline{\,\rmi (z-a_0)^{n}},\\[.5mm]
\ds{\widetilde{h}}_n^{(3)}(z)=\kappa (z-a_0)^{n}-(z-a_0)\,\overline{n(z-a_0)^{n-1}},\\[.5mm]
\ds{\widetilde{h}}_n^{(4)}(z)=\kappa \rmi (z-a_0)^{n}-(z-a_0)\overline{\,\rmi n(z-a_0)^{n-1}}.
\end{cases}
\eeq

\begin{definition} [Modified EMTs using the approximated disk $D$]
Fix $n,m\in\NN$ and $t,s\in\{1,2,3,4\}$. Let $\mathbf{H},\mathbf{F}$ be given by $H={\widetilde{h}}_n^{(t)}(z), F={\widetilde{h}}_{m}^{(s)}(z)$. Set $a^{\alpha}_j, b^{\beta}_k$ to be the coefficients of the Taylor series expansions of $\mathbf{H}, \mathbf{F}$ as in \cref{HF:poly:expan}.
We define
\beq\label{def:Enmts}
{\widetilde{\mathbb{E}}}_{nm}^{(t,s)}(\Omega):=\sum_{\alpha\beta j k}a^\alpha_j b^\beta_k m^j_{\alpha \beta k}(\Omega).
\eeq
\end{definition}

One can easily find that 
\beq\label{modiEMT:int}
{\widetilde{\mathbb{E}}}^{(t,s)}_{nm}(\Omega)
 =\int_{\partial D}\Re\left\{F\,\overline{\varphi}\right\}d\sigma\\
 =\int_{\partial D}\frac{1}{2}\left(F\,\overline{\varphi}+\overline{F}{\varphi}   \right)d\sigma,
 \eeq
where $F(z)={\widetilde{h}}_m^{(s)}(z)$, and let $\varphi$ be the density function corresponding to \cref{eqn:density:phi:psi} with $H(z)={\widetilde{h}}_n^{(t)}(z)$ and the domain $\Omega$. 
Set $t,s=1,2$. It then holds that
\begin{align}\label{def:cnk}
\begin{aligned}
{\widetilde{h}}_n^{(t)}(z)=\mbox{const.}+ \sum_{k=1}^{n}c_{nk}\,h_k^{(t)}(z),\quad {\varphi} = \sum_{k=1}^{n}c_{nk}\, {\varphi}_{k}^{(t)}
\quad\mbox{with }
c_{nk} = \left(-\overline{a_0}\right)^{n-k}\binom{n}{k},
\end{aligned}
\end{align}
where $h_k^{(t)}(z)$ is defined by \cref{def:H:family:center0} and $\varphi_k^{(t)}$ is the density function corresponding to \cref{eqn:density:phi:psi} with $H(z)=h_k^{(t)}(z)$.
Note that ${{h}}_k^{(2)}=-\rmi \,{{h}}_k^{(1)}$.
If follows from \cref{contracted_EMT:integral_form} that for $t,s=1,2$, 
\begin{align}\label{modiEMT:EMT}
\begin{aligned}
\widetilde{\mathbb{E}}^{(t,s)}_{nm}(\Omega)
=\sum_{l=1}^m\sum_{k=1}^n& \Bigg(\Re\left\{ c_{ml}\overline{c_{nk}}\right\}\int_{\p D} \Re\left\{h_l^{(s)}\overline{\varphi_k^{(t)}}\right\}d\sigma
\\&\ + \Im\left\{ c_{ml}\overline{c_{nk}}\right\}\int_{\p D}\Re\left\{\rmi\, h_l^{(s)}\overline{\varphi_k^{(t)}}\right\}d\sigma\Bigg)
\\ 
=\sum_{l=1}^m\sum_{k=1}^n& \left(\Re\left\{ c_{ml}\overline{c_{nk}}\right\} \mathbb{E}_{kl}^{(t,s)}(\Omega)
+\Im\left\{ c_{nk}\overline{c_{ml}}\right\}I_{kl}^{(t,s)}\right)
\end{aligned}
\end{align}
with
$$
I_{kl}^{(t,s)}=
\begin{cases}
-{\mathbb{E}}_{kl}^{(t,2)}(\Omega)&\mbox{ if }s=1,\\ 
{\mathbb{E}}_{kl}^{(t,1)}(\Omega)&\mbox{ if }s=2.
\end{cases}
$$

The modified EMTs of $D$ admit simple expressions as follows. 
Set $H(z)={\widetilde{h}}_n^{(t)}$ and $F(z)= {\widetilde{h}}_m^{(s)}$. On $\p D$, by \cref{cndn:ngeq3:mn}, $\varphi=M_0 n \gamma^n\varphi_{-n}$ for $t=1$ and $\varphi=-\rmi M_0 n \gamma^n\varphi_{-n}$ for $t=2$;  $F(z)=\gamma^{m+1} \varphi_{-m}$ for $s=1$ and $F(z)=-\rmi \gamma^{m+1} \varphi_{-m}$ for $s=2$, where $M_0$ is given by \cref{def:M0}.
By \cref{modiEMT:int}, it holds that
\beq\label{modiEMT:D}
\begin{aligned}
\ds{\widetilde{\mathbb{E}}}^{(1,2)}_{nm}(D)&={\widetilde{\mathbb{E}}}^{(2,1)}_{nm}(D)=0,\\
\ds{\widetilde{\mathbb{E}}}^{(1,1)}_{nm}(D)&={\widetilde{\mathbb{E}}}^{(2,2)}_{nm}(D)=2\pi M_0 n \delta_{nm} \gamma^{n+m}\quad\mbox{for any }n,m\in\NN.
\end{aligned}
\eeq

\subsection{Explicit formulas for the shape perturbation}\label{subsec:explicit:shape}
We regard $\Omega$ as a perturbation of the disk $D=B(a_0,\gamma)$, that is,
\beq\label{Om:disk_peturb}
\p\Omega=\left\{z+\eps h_{\p D}(z) N_{\p D}(z):\, z\in \p D\right\}
\eeq
for some real-valued $C^{1,\gamma}$ function $h_{\p D}$ for some $\gamma\in(0,1]$.
We set $h(\theta)$ by $h_{\p D}(a_0+\gamma e^{\rmi\theta})=\gamma h(\theta)$, $\theta\in[0,2\pi)$.
Being a real-valued function, $h(\theta)$ admits the Fourier series
\begin{align*}
h(\theta)=\widehat{h}_0+\sum_{m=1}^{\infty}\left(\widehat{h}_m e^{\rmi m\theta}+\overline{\widehat{h}_m} e^{-\rmi m\theta}\right),\quad\widehat{h}_m 
= \dfrac{1}{2\pi\gamma} \int_{\partial D}\varphi_{-m}h_{\partial D}d\sigma(\bm{x}).
\end{align*}

To recover $\eps h$, we consider 
\beq\label{Delta_nm}
\Delta_{nm}^{(t,s)}:={\widetilde{\mathbb{E}}}_{nm}^{(t,s)}(\Omega)- {\widetilde{\mathbb{E}}}_{nm}^{(t,s)}(D)\quad\mbox{for }t,s=1,2,
\eeq
where the values of $\Delta_{nm}^{(t,s)}$ can be easily obtained from ${\mathbb{E}}_{nm}^{(t,s)}(\Omega)$ by using \cref{def:cnk,modiEMT:EMT,modiEMT:D}.
By \cref{thm:asymp:interior}, we derive the following theorem. The proof is provided at the end of this subsection. 
\begin{theorem}\label{thm:main:asymptotic}
Let $(\l,\m)$ be the background Lam\'e constants. Let $\Omega$ be a bounded simply connected domain given by \cref{Om:disk_peturb}, whose Lam\'e constants $(\tilde{\l},\tilde{\m})$ are known. For some $\delta>0$, it holds that
\begin{align*}
&\ \Delta_{nm}^{(1,1)}+\Delta_{nm}^{(2,2)}-\rmi (\Delta_{nm}^{(1,2)}-\Delta_{nm}^{(2,1)})\\
&=\, 8(\tilde{\mu}-\mu)nm\gamma^{n+m-1}\,M_1\, \eps \int_{\p D}\varphi_{-n+m}\, h_{\partial D}\,d\sigma(x)+O(\eps^{1+\delta}),\\
&\Delta_{nm}^{(1,1)}-\Delta_{nm}^{(2,2)}+\rmi (\Delta_{nm}^{(1,2)}+\Delta_{nm}^{(2,1)})\\
&=\, 8 (\tilde{\mu}-\mu)nm\gamma^{n+m-1}\, M_1 M_2\, \eps\int_{\p D}  \varphi_{-n-m-2} \,h_{\partial D}\,d\sigma(x)+O(\eps^{1+\delta})
\end{align*}
with
\beq\label{def:M1:M2}
 M_1=\frac{1}{\tilde{\mu}{\a}+\mu\b},\quad M_2= \frac{\b(\mu-\tilde{\mu})}{\tilde{\mu}{\a}+\mu\b}.
\eeq
\end{theorem}

As a direct consequence of \cref{thm:main:asymptotic}, we derive an explicit asymptotic formula for the Fourier coefficients of the shape perturbation function $h$ as the following theorem.
\begin{theorem}\label{thm:main_result}
Let $(\l,\m)$ be the background Lam\'e constants. Let $\Omega$ be a bounded simply connected domain given by \cref{Om:disk_peturb}, whose Lam\'e constants $(\tilde{\l},\tilde{\m})$ are known. Assume that for some $K\in\NN$, ${\mathbb{E}}_{nm}^{(t,s)}(\Omega)$ are given for $1\leq m\leq n\leq K$, $t,s=1,2$.
Then, the boundary of $\Omega$ can be approximated by the parameterized curve:
\begin{align}\label{eqn:main_result}
\theta\mapsto a_0+\gamma e^{\rmi\theta}+2  \Re{\left(\sum_{k=0}^{K-1}\eps\widehat{h}_ke^{\rmi k\theta}\right) } \gamma e^{\rmi\theta},\quad \theta\in[0,2\pi),
\end{align}
where the Fourier coefficients $\widehat{h}_k$ of the shape perturbation function $h(\theta)$ satisfy
\begin{align}\label{hat_h:formula}
\begin{aligned}
\ds\eps\widehat{h}_{n-m}&=
\dfrac{\Delta_{nm}^{(1,1)}+\Delta_{nm}^{(2,2)}-\rmi (\Delta_{nm}^{(1,2)}-\Delta_{nm}^{(2,1)})}{16\pi nm\gamma^{n+m}(\tilde{\mu}-\mu)\,M_1}
+O(\eps^{1+\delta}),
\\[1mm] 
\ds\eps\widehat{h}_{n+m+2}&=
\dfrac{\Delta_{nm}^{(1,1)}-\Delta_{nm}^{(2,2)}+\rmi (\Delta_{nm}^{(1,2)}+\Delta_{nm}^{(2,1)})}{16\pi nm\gamma^{n+m}{(\tilde{\mu}-\mu)\,M_1M_2}}
+O(\eps^{1+\delta})
\end{aligned}
\end{align}
with $M_1,M_2$ given by \cref{def:M1:M2}.
\end{theorem}

\begin{proof}[Proof of \cref{thm:main:asymptotic}]
Let $\gamma$, $w$ and $\varphi_k$ be given as in \cref{subsec:single:expan}. 
We set $u$ to be the complex function corresponding to the solution $\bu$  to \cref{eqn:main:trans} for $\H$ given by \cref{H:p0} as in \cref{subsec:trans:disk}.
We also set $u^e$ and $u^i$ as the restrictions of $u$ to the exterior and interior of $D$, respectively. 
To indicate the dependence of $u^e$ and $u^i$ on $(n,q)$, we also write $u^e(\cdot;n,q)$ and $u^i(\cdot;n,q)$.

As the novel idea to deal with the strain terms in the asymptotic formula in \cref{thm:asymp:interior}, we introduce the matrix $$\mathbb{J}:=\bmat{1 & \rmi \\ \rmi &-1}.$$
Note that $\operatorname{tr}(\mathbb{J})=\operatorname{tr}(\overline{\mathbb{J}})=0$ and
\begin{gather*}
\mathbb{J}:\mathbb{I}_2=\overline{\mathbb{J}}:\mathbb{I}_2=0,\quad
\mathbb{J}:\mathbb{J}=\overline{\mathbb{J}}:\overline{\mathbb{J}}=0,\quad
\mathbb{J}:\overline{\mathbb{J}}=4,
\end{gather*}
where $\mathbb{I}_2$ is the identity matrix.

For a vector-valued function $\bm{v}=(v_1,v_2)$ and $v=v_1+\rmi v_2$, we have
\begin{align*}
\widehat{\n}v:=\widehat{\n}\bm{v}=
\frac{1}{2}
\bmat{\ds\left(\pd{}{z}+\pd{}{\overline{z}}\right)(v+\overline{v}) & \ds \rmi \left(\pd{}{z}\overline{v}-\pd{}{\overline{z}}v\right)\\
\ds \rmi\left(\pd{}{z}\overline{v}-\pd{}{\overline{z}}v\right) & \ds \left(\pd{}{z}-\pd{}{\overline{z}}\right)(v-\overline{v})}.
\end{align*}
Hence, for $c\in\CC$ and $k\in\ZZ$, it holds that on $\p D$,
\begin{align*}
\widehat{\n}\,\overline{c(z-a_0)^k}&= \frac{k\gamma^{k}}{2}\left(c\varphi_{k-1}
\mathbb{J}+\overline{c}\varphi_{-k+1}\overline{\mathbb{J}}\right),\\
\widehat{\n}\,\overline{c}(z-a_0)^{k}&= \frac{k\gamma^{k}}{2}\left(\overline{c}\varphi_{k-1} + c\varphi_{-k+1}\right)\mathbb{I}_2,\\
\widehat{\n}\,(z-a_0)\,\overline{c(z-a_0)^k}&= \frac{\gamma^{k+1}}{2}\left(c\varphi_k+\overline{c}\varphi_{-k}\right)\mathbb{I}_2+
\frac{ k\gamma^{k+1}}{2}\left(
c\varphi_{k-2} \mathbb{J}+\overline{c}\varphi_{-k+2}\overline{\mathbb{J}}
\right).
\end{align*}

On $\p D$, by \cref{cndn:ngeq3:mn,Scal:int:p0,Scal:ext:p0}, it holds that
\begin{align*}
\widehat{\n}u^i(\cdot;n,q)
&=\frac{1}{\tilde{\mu}{\a}+\mu\b}\dfrac{n\gamma^{n}}{2}\left(q\varphi_{n-1}\mathbb{J}+\overline{q}\varphi_{-n+1}\overline{\mathbb{J}}\right),\\
\widehat{\n}u^e(\cdot;n,q)
&=\frac{n\gamma^{n}}{2}(q\varphi_{n-1}\mathbb{J}+\overline{q}\varphi_{-n+1}\overline{\mathbb{J}})
+(\alpha-\beta)\frac{\tilde{\mu}-\mu}{\tilde{\mu}\alpha+\mu\beta}\frac{n\gamma^{n}}{2}(\overline{q}\varphi_{-n-1} + q\varphi_{n+1})\mathbb{I}_2
\\&\ -\frac{\beta(\tilde{\mu}-\mu)}{\tilde{\mu}\alpha+\mu\beta}\frac{n\gamma^n}{2}(\overline{q}\varphi_{-n-3}\mathbb{J}+q\varphi_{n+3}\overline{\mathbb{J}}).
\end{align*}
By \cref{C:symm:a}, it follows that
\beq\label{widehat:uie}
\begin{aligned}
(\CC_1-\CC_0)\widehat{\n}u^i(\cdot;n,q)&=2(\tilde{\mu}-\mu)\widehat{\n}u^i(\cdot;n,q),\\
(\CC_1-\CC_0)\widehat{\n}u^e(\cdot;n,q)&=\frac{\tilde{\lambda}-\lambda}{2\mu+\lambda}\frac{\tilde{\mu}-\mu}{\tilde{\mu}\alpha+\mu\beta}n\gamma^{n}(\overline{q}\varphi_{-n-1}+q\varphi_{n+1})\mathbb{I}_2
\\&\ \ 
+2(\tilde{\mu}-\mu)\widehat{\n}u^{e}(\cdot;n,q).
\end{aligned}
\eeq
Plugging in $q=1,i$, we obtain
\begin{align*}
\widehat{\n}u^i(\cdot;n,1)
&=M_1\,\frac{n\gamma^{n}}{2}\left(\varphi_{n-1}\mathbb{J}+\varphi_{-n+1}\overline{\mathbb{J}}\right),\\
\widehat{\n}u^i(\cdot;n,\rmi)
&=\rmi M_1\, \frac{n\gamma^{n} }{2}\left(\varphi_{n-1}\mathbb{J}-\varphi_{-n+1}\overline{\mathbb{J}}\right),\\
\widehat{\n}u^e(\cdot;n,1)
&=\frac{n\gamma^{n}}{2}(\varphi_{n-1}\mathbb{J}+\varphi_{-n+1}\overline{\mathbb{J}})
+M_2\frac{n\gamma^{n}}{2}(\varphi_{-n-3}\mathbb{J}+\varphi_{n+3}\overline{\mathbb{J}})
\\&\  \ +\mbox{other term}*\mathbb{I}_2,
\\
\widehat{\n}u^e(\cdot;n,\rmi)
 &=\rmi \frac{n\gamma^{n}}{2}(\varphi_{n-1}\mathbb{J}-\varphi_{-n+1}\overline{\mathbb{J}})
+\rmi M_2\frac{n\gamma^{n}}{2}(-\varphi_{-n-3}\mathbb{J}+\varphi_{n+3}\overline{\mathbb{J}})
\\&\ \ +\mbox{other term}*\mathbb{I}_2
\end{align*}
with the constants $M_1,M_2$ given by \cref{def:M1:M2}.

Applying \cref{widehat:uie}, we obtain
\begin{align*}
I:=&(\CC_1-\CC_0)\widehat{\n}u^i(\cdot;n,1):\widehat{\n}u^e(\cdot;m,1)
=(\CC_1-\CC_0)\widehat{\n}u^e(\cdot;n,1):\widehat{\n}u^i(\cdot;m,1)\\
=&\,2(\tilde{\mu}-\mu)nm\gamma^{n+m-1}\Big[M_1 \left(\varphi_{n-m}+\varphi_{-n+m}\right)
+M_1M_2\left(\varphi_{n+m+2}+\varphi_{-n-m-2}\right)\Big],\\[2mm]
II:=&(\CC_1-\CC_0)\widehat{\n}u^i(\cdot;n,1):\widehat{\n}u^e(\cdot;m,\rmi)
=(\CC_1-\CC_0)\widehat{\n}u^e(\cdot;n,1):\widehat{\n}u^i(\cdot;m,\rmi)\\
=&\,2(\tilde{\mu}-\mu)nm\gamma^{n+m-1}
\Big[\rmi M_1\left(-\varphi_{n-m}+\varphi_{-n+m}\right)+\rmi M_1M_2\left(\varphi_{n+m+2}-\varphi_{-n-m-2}\right)\Big]
\end{align*}
and
\begin{align*}
III:=&(\CC_1-\CC_0)\widehat{\n}u^i(\cdot;n,\rmi):\widehat{\n}u^e(\cdot;m,1)
=(\CC_1-\CC_0)\widehat{\n}u^e(\cdot;n,\rmi):\widehat{\n}u^i(\cdot;m,1)\\
=&\, 2 (\tilde{\mu}-\mu)nm\gamma^{n+m-1}
\Big[\rmi M_1\left(\varphi_{n-m}-\varphi_{-n+m}\right)+\rmi M_1M_2\left(\varphi_{n+m+2}-\varphi_{-n-m-2}\right)\Big],\\[2mm]
IV:=&(\CC_1-\CC_0)\widehat{\n}u^i(\cdot;n,\rmi):\widehat{\n}u^e(\cdot;m,\rmi)
=(\CC_1-\CC_0)\widehat{\n}u^e(\cdot;n,\rmi):\widehat{\n}u^i(\cdot;m,\rmi)\\
=&\, 2 (\tilde{\mu}-\mu)nm\gamma^{n+m-1}
\Big[M_1\left(\varphi_{n-m}+\varphi_{-n+m}\right)-M_1M_2\left(\varphi_{n+m+2}+\varphi_{-n-m-2}\right)\Big].
\end{align*}
One can easily find that
\beq\label{density_relation}
\begin{aligned}
I+IV-\rmi (II-III)
&= 8(\tilde{\mu}-\mu)nm\gamma^{n+m-1}\,M_1\,\varphi_{-n+m},\\
I-IV+\rmi (II+III)
&=8(\tilde{\mu}-\mu)nm\gamma^{n+m-1}\,M_1M_2\,\varphi_{-n-m-2}.
\end{aligned}
\eeq
In view of \cref{thm:asymp:interior}, $I$ corresponds to the modified EMTs $\widetilde{\mathbb{E}}_{nm}^{(1,1)}$. Similarly, $II$, $III$ and $IV$ correspond to $\widetilde{\mathbb{E}}_{nm}^{(1,2)}$, $\widetilde{\mathbb{E}}_{nm}^{(2,1)}$ and $\widetilde{\mathbb{E}}_{nm}^{(2,2)}$, respectively.
From \cref{density_relation}, we derive that
\beq\label{asymp:int:1}
\begin{aligned}
&\ \left(\widetilde{\mathbb{E}}_{nm}^{(1,1)}+\widetilde{\mathbb{E}}_{nm}^{(2,2)}-\rmi (\widetilde{\mathbb{E}}_{nm}^{(1,2)}-\widetilde{\mathbb{E}}_{nm}^{(2,1)})\right)(\Omega)
-\left(\widetilde{\mathbb{E}}_{nm}^{(1,1)}+\widetilde{\mathbb{E}}_{nm}^{(2,2)}-\rmi (\widetilde{\mathbb{E}}_{nm}^{(1,2)}-\widetilde{\mathbb{E}}_{nm}^{(2,1)})\right)(D)\\
&=\eps\int_{\p D} 8(\tilde{\mu}-\mu)nm\gamma^{n+m-1}\,M_1 \,\varphi_{-n+m}\, h_{\partial D}\,d\sigma(x)+O(\eps^{1+\delta})
\end{aligned}
\eeq
and
\beq\label{asymp:int:2}
\begin{aligned}
&\ \left(\widetilde{\mathbb{E}}_{nm}^{(1,1)}-\widetilde{\mathbb{E}}_{nm}^{(2,2)}+\rmi (\widetilde{\mathbb{E}}_{nm}^{(1,2)}+\widetilde{\mathbb{E}}_{nm}^{(2,1)})\right)(\Omega)
-\left(\widetilde{\mathbb{E}}_{nm}^{(1,1)}-\widetilde{\mathbb{E}}_{nm}^{(2,2)}+\rmi (\widetilde{\mathbb{E}}_{nm}^{(1,2)}+\widetilde{\mathbb{E}}_{nm}^{(2,1)})\right)(D)\\
&=\eps\int_{\p D} 8(\tilde{\mu}-\mu)nm\gamma^{n+m-1}\, M_1 M_2 \,\varphi_{-n-m-2} \,h_{\partial D}\,d\sigma(x)+O(\eps^{1+\delta}).
\end{aligned}
\eeq
Hence, we complete the proof of \cref{thm:main:asymptotic}.
\end{proof}

\section{Numerical computation}\label{sec:numerical}
We show numerical results for the recovery of an unknown elastic inclusion, $\Omega$, based on the analytic reconstruction methods provided in \cref{sec:analytic:recovery}, assuming that the Lam\'{e} constants $\lambda,\mu,\tlambda,\tmu$ are known.

All example inclusions have a smooth boundary. The background parameters are fixed to be $(\lambda,\mu)=(1.5,1.2)$ and the inclusion parameters $\tlambda,\tmu$ are given differently for each example. The values of $\tlambda,\tmu$ are described in figures of numerical results.
In each simulation, we use the contracted EMTs up to a given order, denoted as $\mbox{Ord}\in\NN$. In other words,
\begin{align}\label{data:EMT:Ord}
\left\{\mathbb{E}^{(t,s)}_{nm}(\Omega)\,:\,1\leq n,m\leq \text{Ord},\ t,s=1,2\right\}.
\end{align}
To acquire the values of the EMTs for the example inclusions, we solve the boundary integral equation \cref{eqn:density:phi:psi} and evaluate the integral in \cref{contracted_EMT:integral_form} by employing the Nystr\"om discretization method for the integrals. 

We recover $\Omega$ from \cref{data:EMT:Ord} by the following two-step procedure.
\begin{itemize}
\item {\textbf{Step 1.}} We find a disk $D=B(a_0^{\text{rec}},\gamma^{\text{rec}})$ whose values of $\mathbb{E}^{(1,1)}_{11},\mathbb{E}^{(1,1)}_{12},\mathbb{E}^{(1,2)}_{12}$ coincide with those of $\Omega$, respectively. In other words, by using  \cref{app_disk}, we have
\beq\label{app_disk:recon}
\begin{aligned}
\gamma^{\text{rec}}&=\frac{1}{\sqrt{2\pi M_0}}\sqrt{\mathbb{E}^{(1,1)}_{11}(\Omega)},\quad
a_0^{\text{rec}}=\frac{1}{4\pi\gamma^2 M_0}\left( \mathbb{E}^{(1,1)}_{12}(\Omega)-\rmi\, \mathbb{E}^{(1,2)}_{12}(\Omega)\right)
\end{aligned}
\eeq
with $M_0$ given by \cref{def:M0}, which is a constant depending on $\lambda,\mu,\tlambda,\tmu$.

\item {\textbf{Step 2.}} 
We obtain finer details of the shape of $\Omega$ by taking the image of
\beq\label{step2:algo}
\theta\mapsto a_0^{\text{rec}}+\gamma^{\text{rec}} e^{\rmi\theta}+2 \, \Re{\left(\sum_{k=0}^{\text{Ord}-1}\eps\widehat{h}_ke^{\rmi k\theta}\right) } \gamma^{\text{rec}} e^{\rmi\theta},\quad\theta\in[0,2\pi),
\eeq
where $\eps \widehat{h}_k$ is determined by \cref{hat_h:formula} with $n=k+1$ and $m=1$ and $\gamma^{\text{rec}}$ in the place of $\gamma$, neglecting the $O(\eps^{1+\delta})$ term.

\end{itemize}

In all figures, the gray curve represents the original shape of an example inclusion, while the black curve represents the reconstructed inclusion.

\begin{example}[Starfish-shaped domain] 
In \cref{figure:Star}, we consider a starfish-shaped inclusion whose boundary is given by
$$
z(\theta) = a_0+e^{\rmi\theta} + 2 \, \Re{\left( 8^{-1}e^{5\rmi\theta}\right)} e^{\rmi\theta}\quad\mbox{with }a_0=0, -0.9+1.2\rmi.
$$
For this example, we use the EMTs with multiplicative Gaussian noise, that is, \begin{align}
\label{EMT:noise}
 \mathbb{E}^{(t,s)}_{nm}
\left( 1 + \mathcal{N}\left( 0,\sigma_{\text{noise}}^2\right) \right) \quad\mbox{for each }n,m,t,s,
\end{align}
where $\mathcal{N}$ means the normal distribution. 
We illustrate the reconstruction results obtained by one sample set of EMTs with noise generated by \cref{EMT:noise}. 
The first row of \cref{figure:Star} is the case of $a_0=0$ and the second row is the case of $a_0=-0.9+1.2\rmi$.
Note that \cref{step2:algo} in Step 2 uses ${\widetilde{\mathbb{E}}}_{nm}^{(t,s)}(\Omega)$, on which the error in $a_0$ has a stronger effect for a larger magnitude of $a_0$. Hence, the proposed method is more sensitive to the measurement noise when $a_0$ has a larger magnitude, as is shown in \cref{figure:Star}. 
\end{example}

\begin{example}[Kite-shaped domain] 
In \cref{figure:Kite}, we consider a kite-shaped inclusion whose boundary is given by $z(\theta) =0.6+0.8\rmi + e^{\rmi \theta} + 0.65\cos(2\theta).$
The EMTs without noise up to $\text{Ord}=3,4,6,8$ are used.
\end{example}

\begin{example}[Ellipse] 
We consider an ellipse-shaped inclusion whose boundary is given by a parametrization:
$z(\theta) = 4\cos(\theta) + \rmi\sin(\theta)$. 
\Cref{figure:Ellipse} shows that the reconstruction performance varies depending on the Lam\'{e} constants even when the inclusion has the same shape. 
\end{example}

\begin{figure}[h!]
\begin{subfigure}{\linewidth}
\centering
\captionsetup{justification=centering}
\begin{minipage}{0.3\linewidth}
\subcaption*{\normalsize $\sigma_{\text{noise}}^2=0.01$}
\end{minipage}
\begin{minipage}{0.3\linewidth}
\subcaption*{\normalsize $\sigma_{\text{noise}}^2=0.05$}
\end{minipage}
\begin{minipage}{0.3\linewidth}
\subcaption*{\normalsize $\sigma_{\text{noise}}^2=0.1$}
\end{minipage}
\end{subfigure}
\begin{subfigure}{\linewidth}
\centering
\captionsetup{justification=centering}
\begin{minipage}{0.3\linewidth}
\includegraphics[width=.98\linewidth,trim=70 30 50 15, clip]{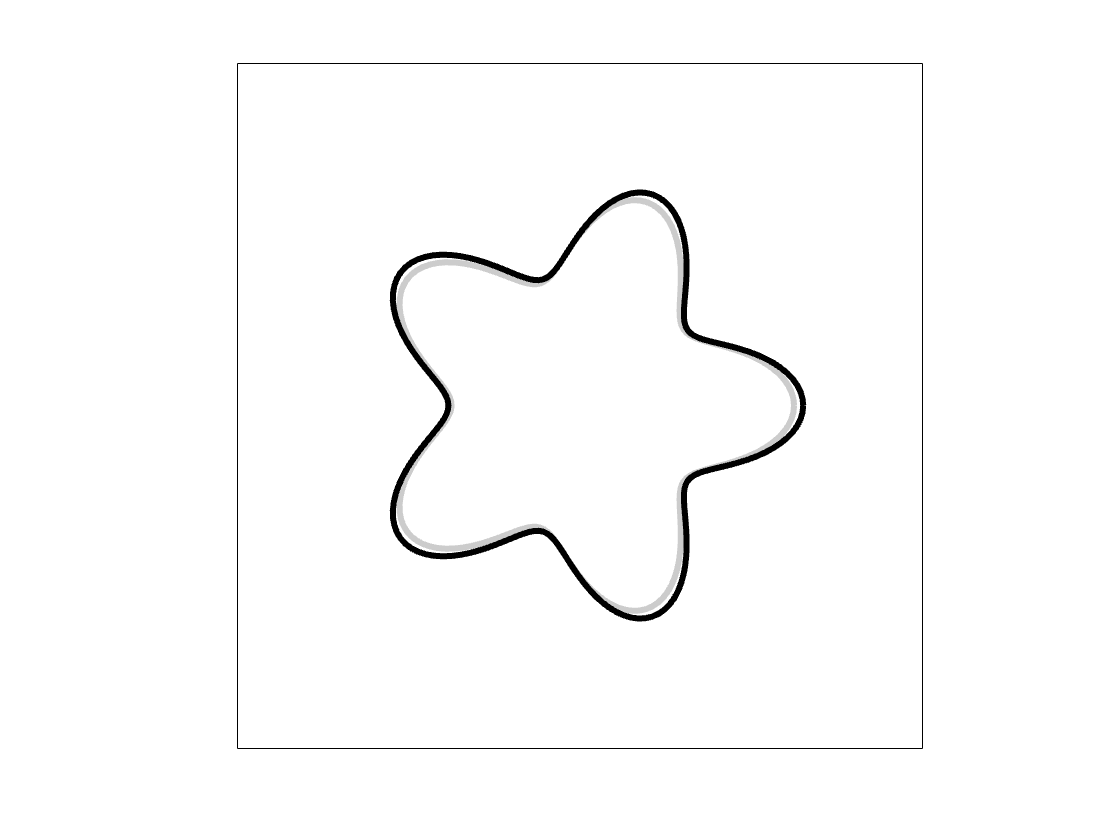}
\end{minipage}
\begin{minipage}{0.3\linewidth}
\includegraphics[width=.98\linewidth,trim = 70 30 50 15, clip]{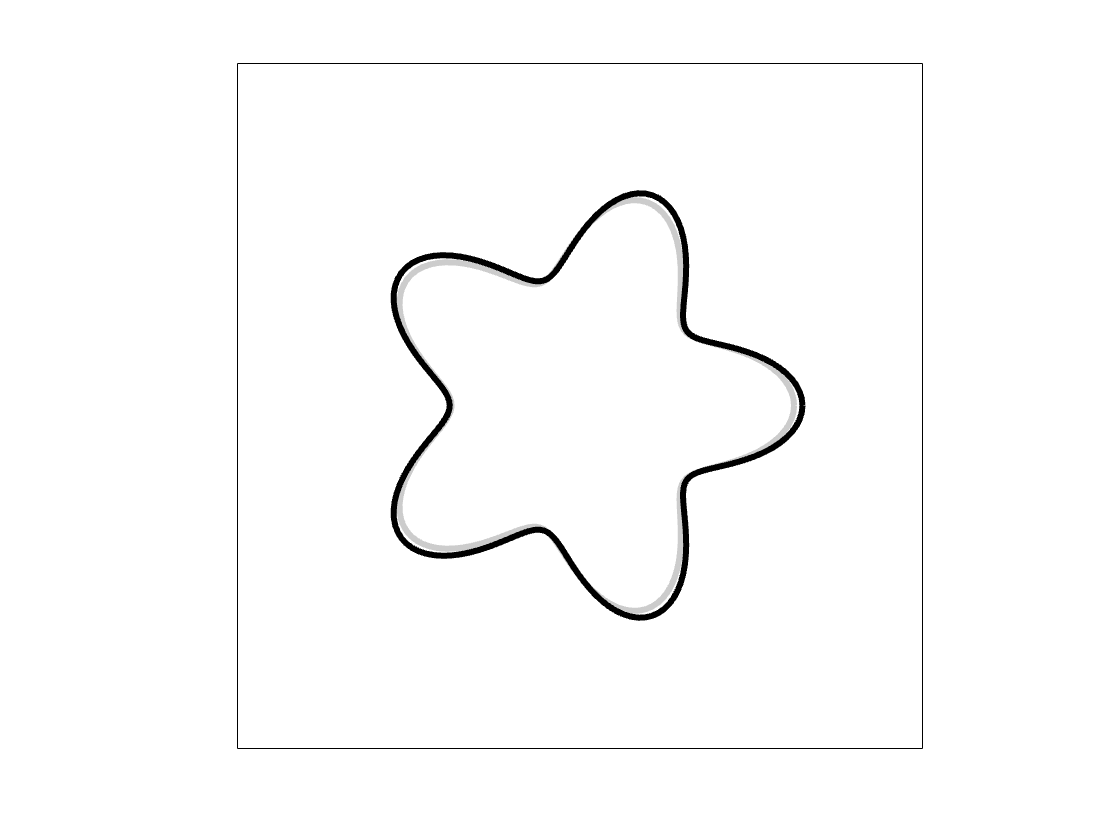}
\end{minipage}
\begin{minipage}{0.3\linewidth}
\includegraphics[width=.98\linewidth,trim = 70 30 50 15, clip]{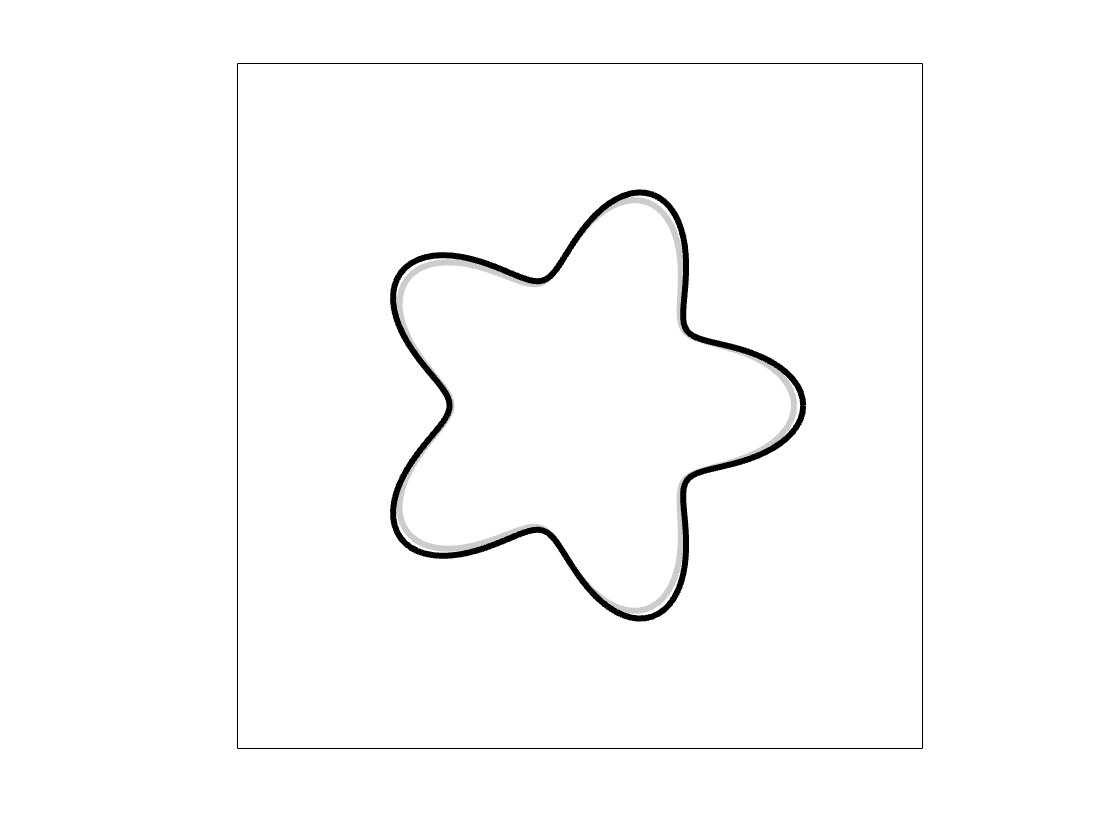}
\end{minipage}
\end{subfigure}
\begin{subfigure}{\linewidth}
\centering
\vskip -3.5mm
\captionsetup{justification=centering}
\begin{minipage}{0.3\linewidth}
\subcaption*{$ {a}_0^{\text{rec}} =  0.0216$\\${\gamma}^{\text{rec}}=1.0257 $}
\end{minipage}
\begin{minipage}{0.3\linewidth}
\subcaption*{${a}_0^{\text{rec}} = 0.0231$\\${\gamma}^{\text{rec}}=1.0318$}
\end{minipage}
\begin{minipage}{0.3\linewidth}
\subcaption*{${a}_0^{\text{rec}} = 0.0198$\\${\gamma}^{\text{rec}}=1.0290$}
\end{minipage}
\end{subfigure}
\begin{subfigure}{\linewidth}
\centering
\captionsetup{justification=centering}
\begin{minipage}{0.3\linewidth}
\includegraphics[width=.98\linewidth,trim=70 30 50 15, clip]{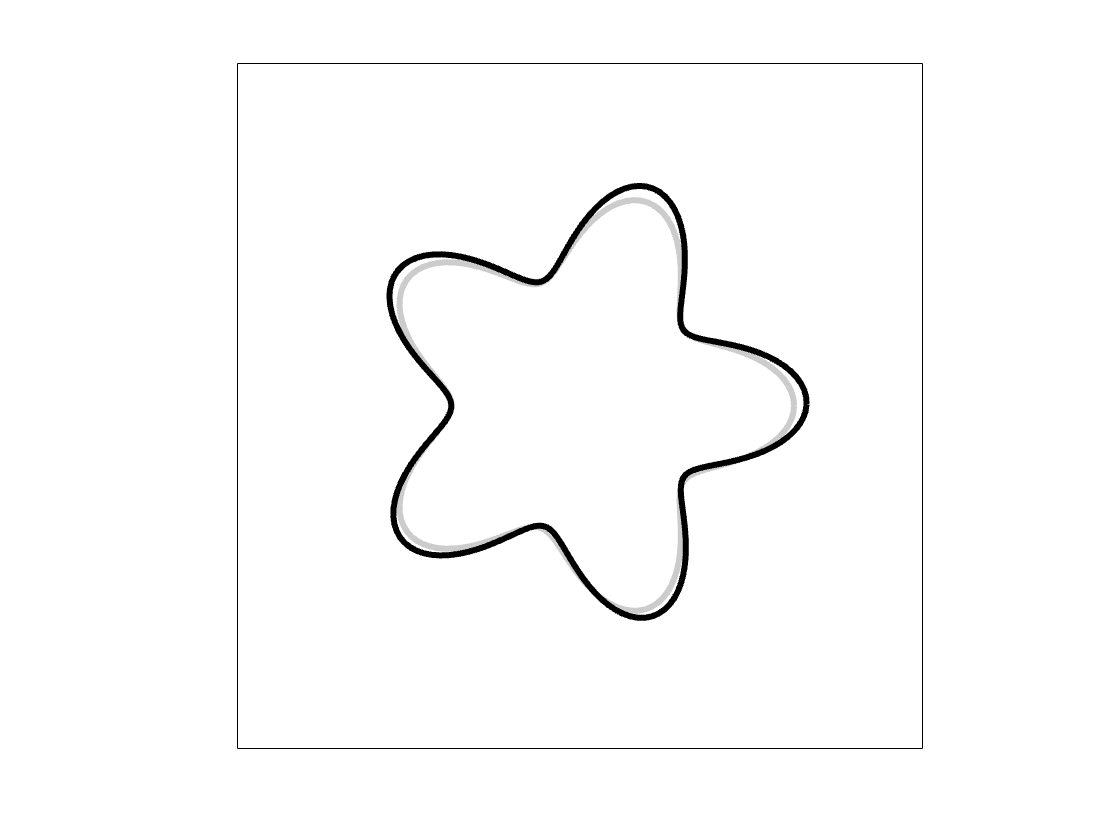}
\end{minipage}
\begin{minipage}{0.3\linewidth}
\includegraphics[width=.98\linewidth,trim = 70 30 50 15, clip]{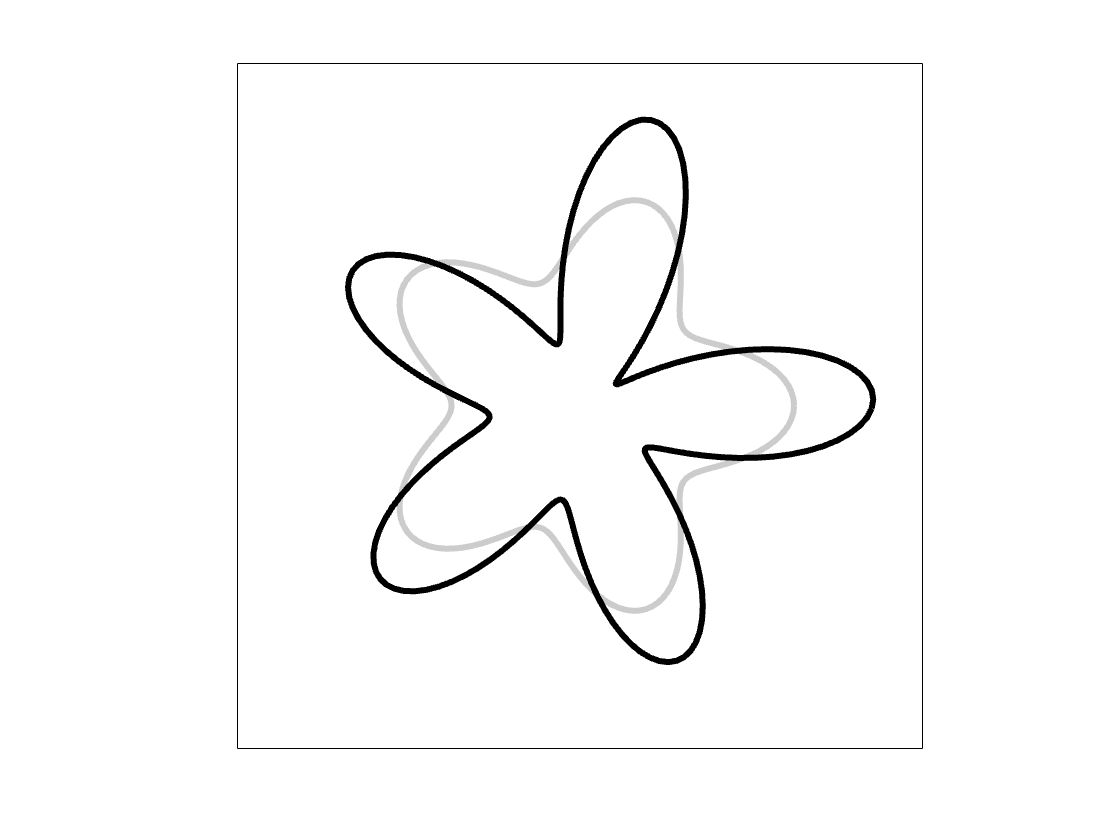}
\end{minipage}
\begin{minipage}{0.3\linewidth}
\includegraphics[width=.98\linewidth,trim = 70 30 50 15, clip]{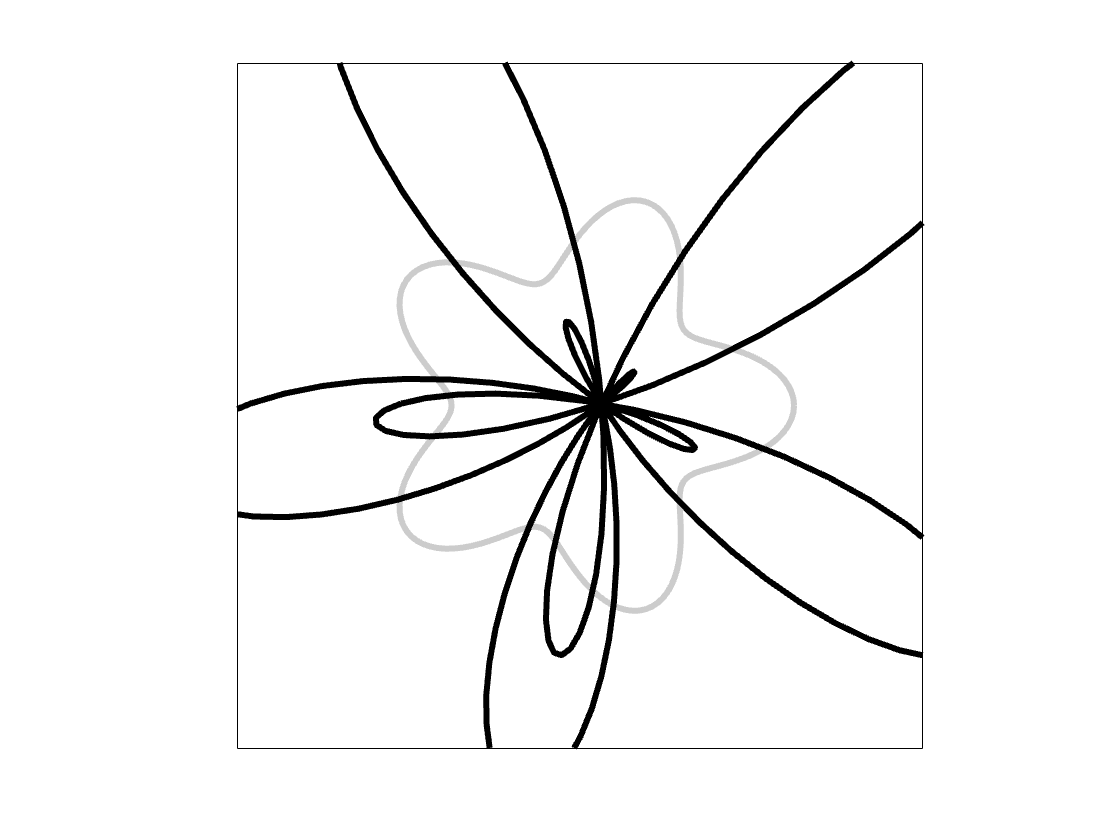}
\end{minipage}
\end{subfigure}
\begin{subfigure}{\linewidth}
\centering
\vskip -3.5mm
\captionsetup{justification=centering}
\begin{minipage}{0.3\linewidth}
\subcaption*{${a}_0^{\text{rec}} =    -0.8825 + 1.2076\rmi $ \\ $ {\gamma}^{\text{rec}}=1.0295 $}
\end{minipage}
\begin{minipage}{0.3\linewidth}
\subcaption*{${a}_0^{\text{rec}} =    -0.8647 + 1.1815\rmi$\\${\gamma}^{\text{rec}}=1.0343$}
\end{minipage}
\begin{minipage}{0.3\linewidth}
\subcaption*{${a}_0^{\text{rec}}= -0.7746 + 1.2110\rmi$\\${\gamma}^{\text{rec}}=1.0748$}
\end{minipage}
\end{subfigure}
\caption{Starfish-shaped inclusion. The Lam\'e constants for $\Om$ is $(\tilde{\lambda},\tilde{\mu}) = (0.6,0.4)$. The EMTs up to $\text{Ord}= 6$ are used. 
The first row is the case of $a_0=0$ and the second row is the case of $a_0=-0.9+1.2\rmi$. 
The reconstruction is more sensitive to the noise in EMTs when $a_0$ has a larger magnitude. }
\label{figure:Star}
\end{figure}

\begin{figure}[h!]
\begin{subfigure}{\linewidth}
\centering
\captionsetup{justification=centering}
\begin{minipage}{0.1\linewidth}
\subcaption*{}
\end{minipage}
\begin{minipage}{0.28\linewidth}
\subcaption*{\normalsize $(\tilde{\lambda},\tilde{\mu})=(1.8,1.5)$}
\end{minipage}
\begin{minipage}{0.28\linewidth}
\subcaption*{\normalsize$(\tilde{\lambda},\tilde{\mu})=(1,0.8)$}
\end{minipage}
\begin{minipage}{0.28\linewidth}
\subcaption*{\normalsize $(\tilde{\lambda},\tilde{\mu})=(0.6,0.4)$}
\end{minipage}
\end{subfigure}\\
\begin{subfigure}{\linewidth}
\centering
\captionsetup{justification=centering}
\begin{minipage}{0.1\linewidth}
\subcaption*{Ord$=3$}
\end{minipage}
\begin{minipage}{0.28\linewidth}
\includegraphics[width=0.98\linewidth,trim=70 30 50 15, clip]{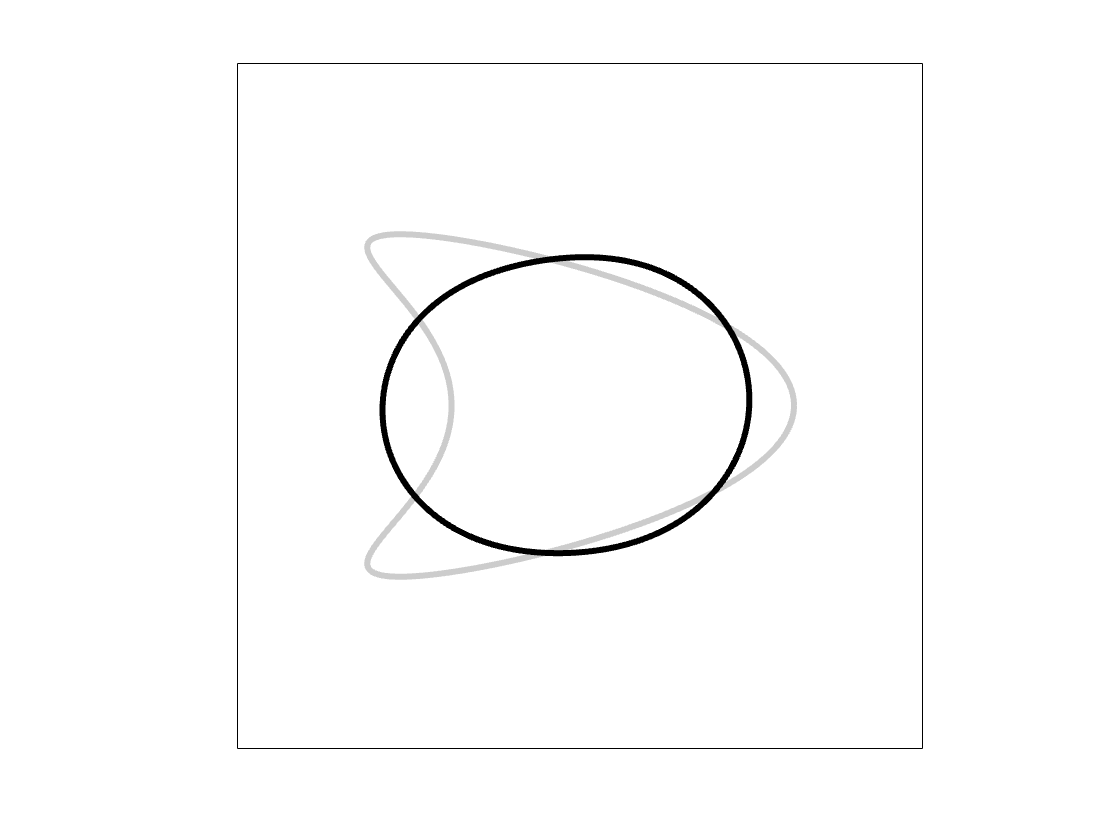}
\end{minipage}
\begin{minipage}{0.28\linewidth}
\includegraphics[width=0.98\linewidth,trim = 70 30 50 15, clip]{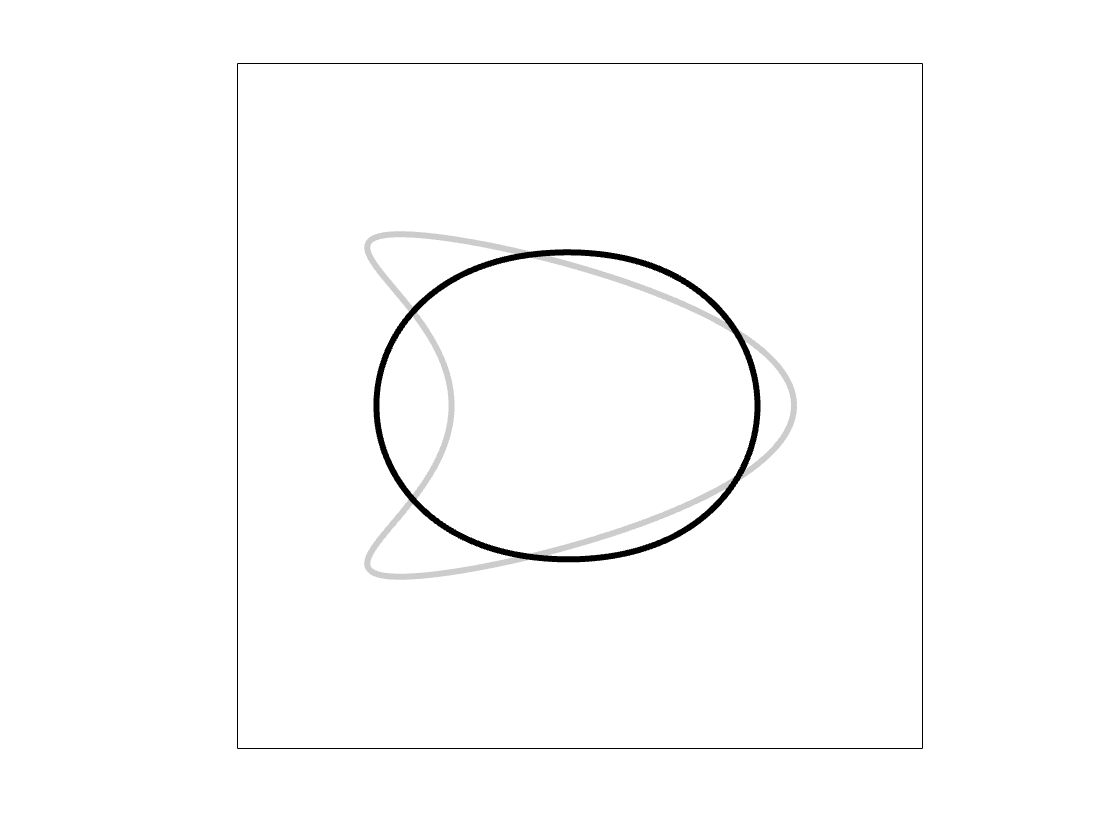}
\end{minipage}
\begin{minipage}{0.28\linewidth}
\includegraphics[width=0.98\linewidth,trim = 70 30 50 15, clip]{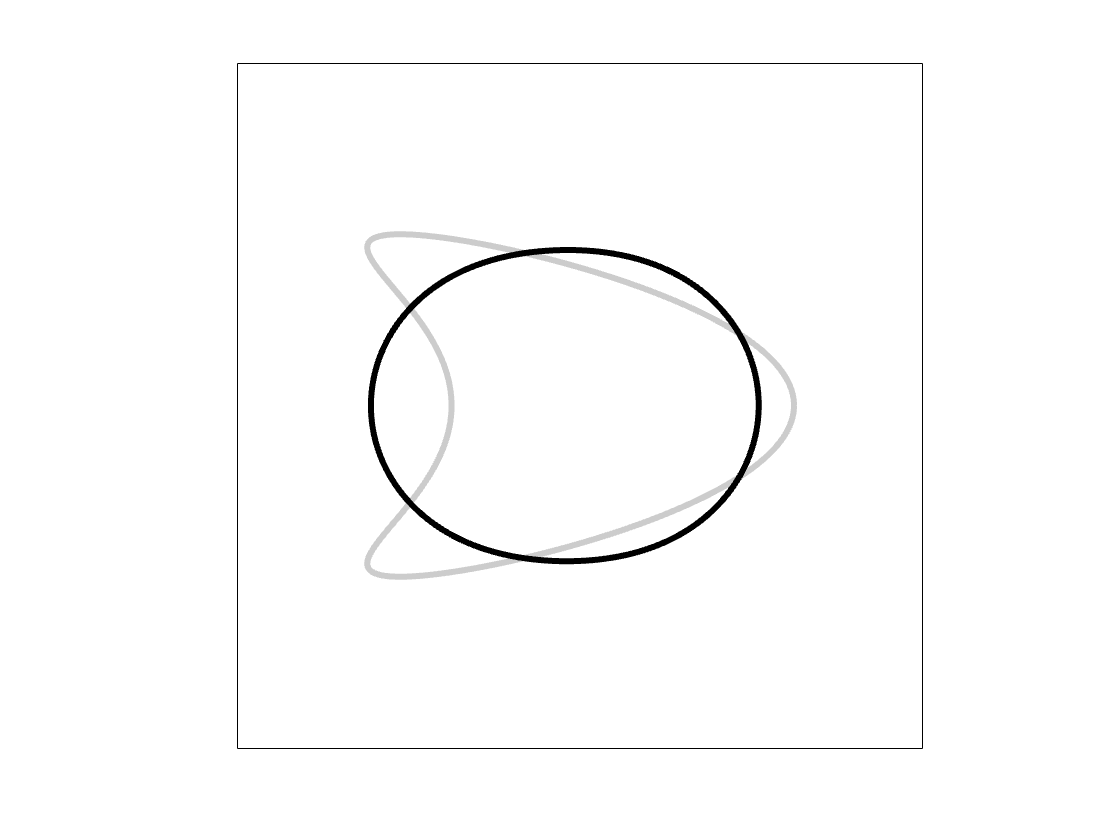}
\end{minipage}
\end{subfigure}\\
\begin{subfigure}{\linewidth}
\centering
\captionsetup{justification=centering}
\begin{minipage}{0.1\linewidth}
\subcaption*{Ord$=4$}
\end{minipage}
\begin{minipage}{0.28\linewidth}
\includegraphics[width=0.98\linewidth,trim=70 30 50 15, clip]{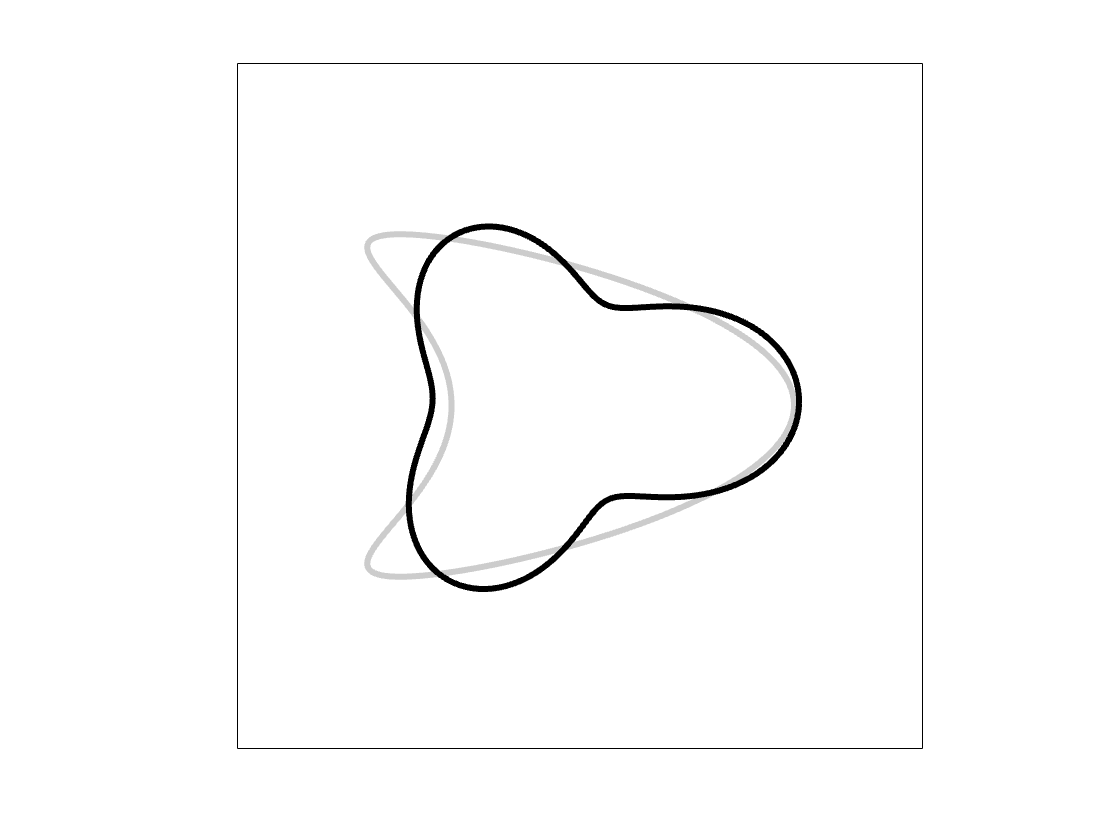}
\end{minipage}
\begin{minipage}{0.28\linewidth}
\includegraphics[width=0.98\linewidth,trim = 70 30 50 15, clip]{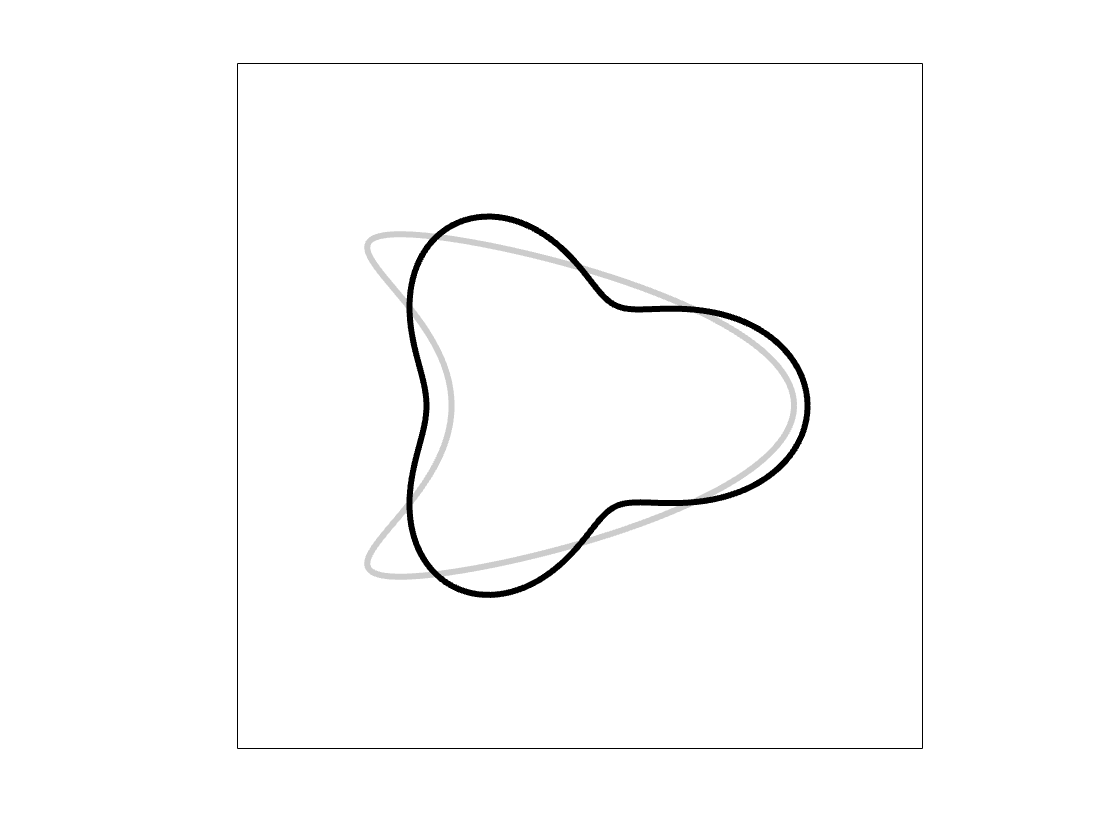}
\end{minipage}
\begin{minipage}{0.28\linewidth}
\includegraphics[width=0.98\linewidth,trim = 70 30 50 15, clip]{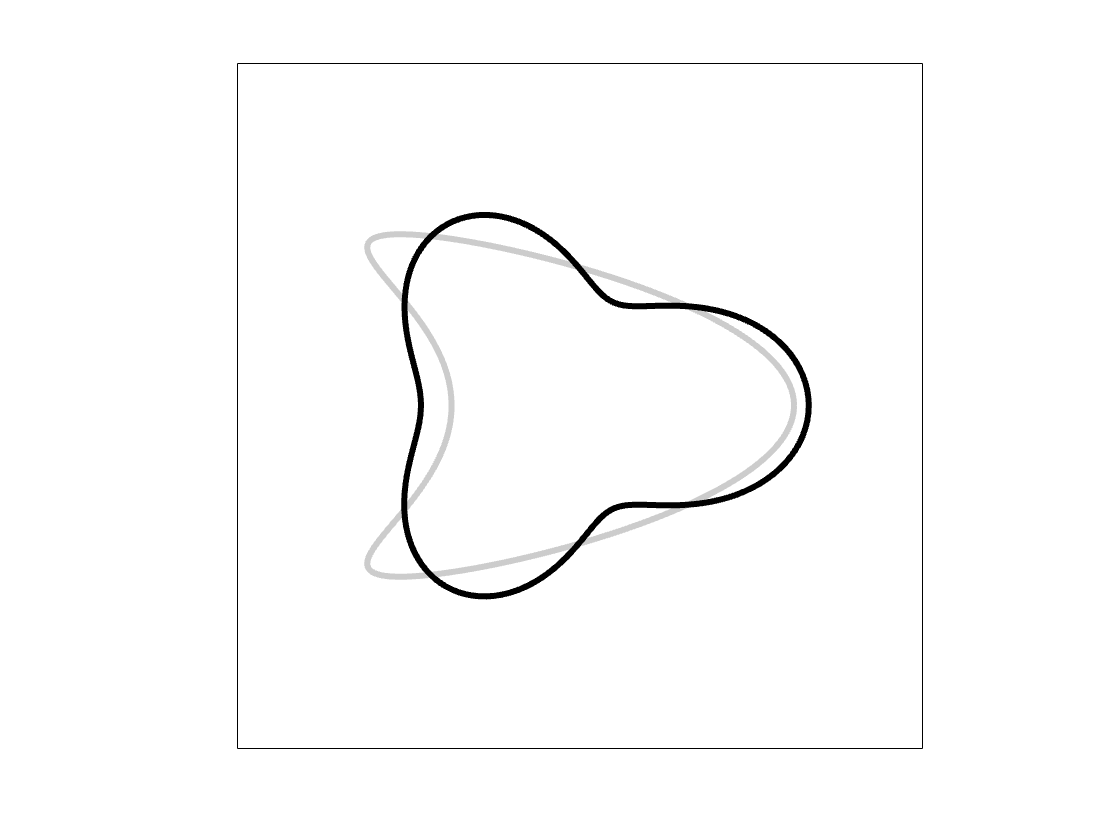}
\end{minipage}
\end{subfigure}\\
\begin{subfigure}{\linewidth}
\centering
\captionsetup{justification=centering}
\begin{minipage}{0.1\linewidth}
\subcaption*{Ord$=6$}
\end{minipage}
\begin{minipage}{0.28\linewidth}
\includegraphics[width=0.98\linewidth,trim=70 30 50 15, clip]{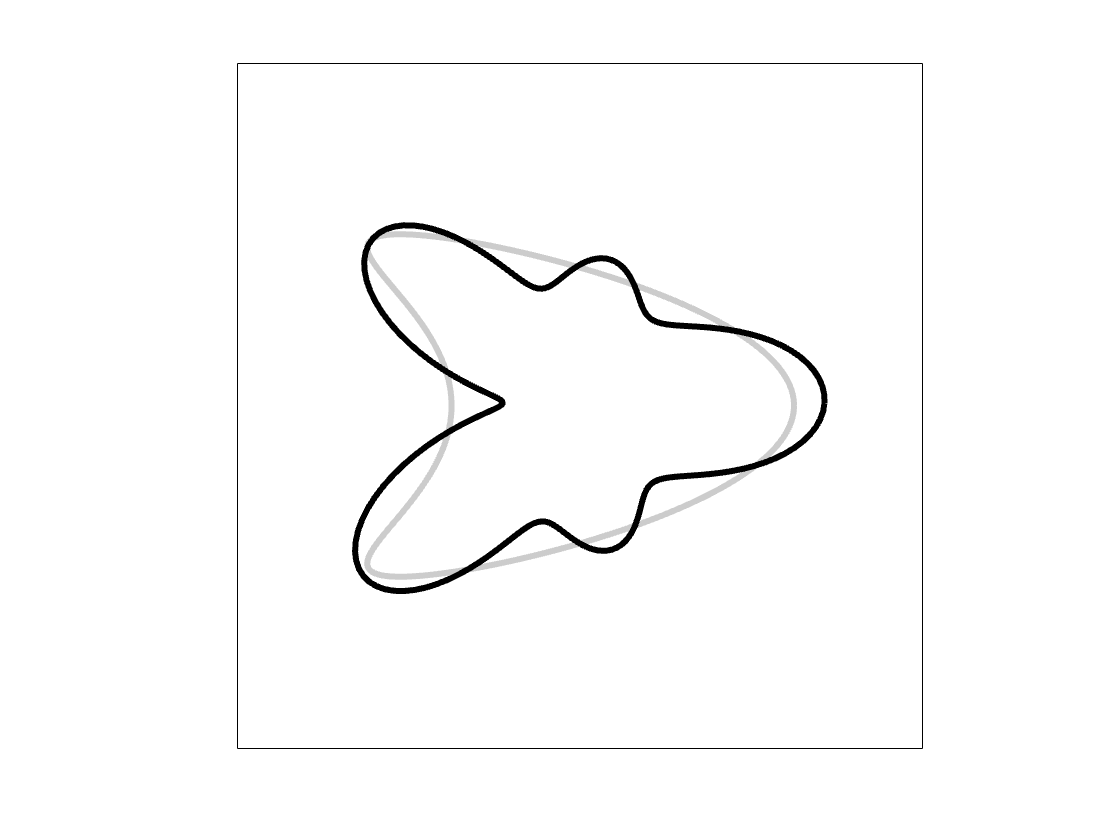}
\end{minipage}
\begin{minipage}{0.28\linewidth}
\includegraphics[width=0.98\linewidth,trim = 70 30 50 15, clip]{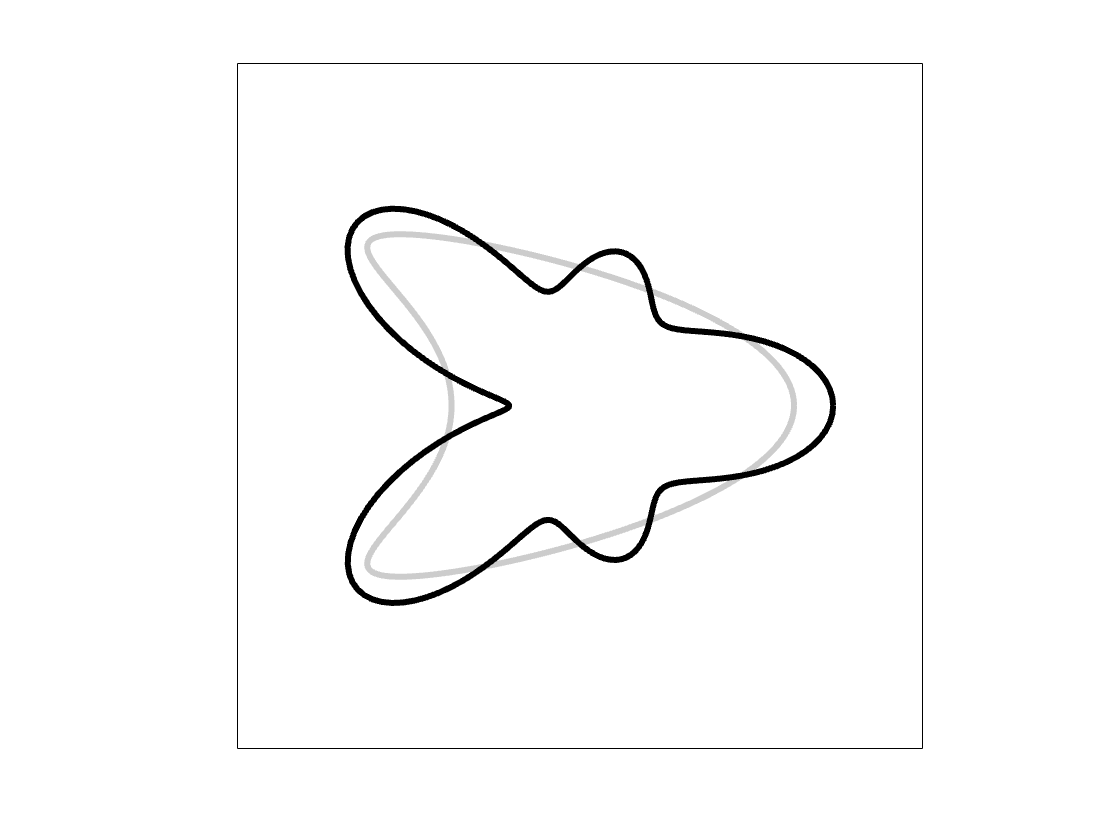}
\end{minipage}
\begin{minipage}{0.28\linewidth}
\includegraphics[width=0.98\linewidth,trim = 70 30 50 15, clip]{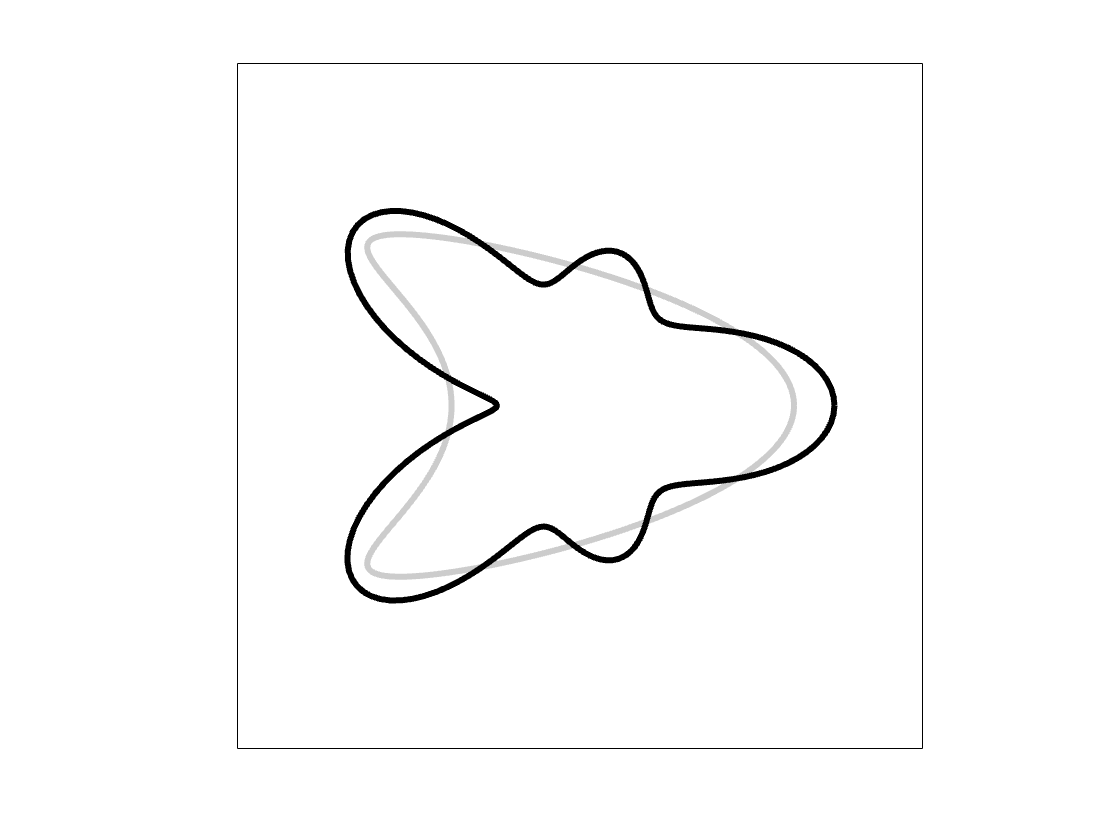}
\end{minipage}
\end{subfigure}\\
\begin{subfigure}{\linewidth}
\centering
\captionsetup{justification=centering}
\begin{minipage}{0.1\linewidth}
\subcaption*{Ord$=8$}
\end{minipage}
\begin{minipage}{0.28\linewidth}
\includegraphics[width=0.98\linewidth,trim=70 30 50 15, clip]{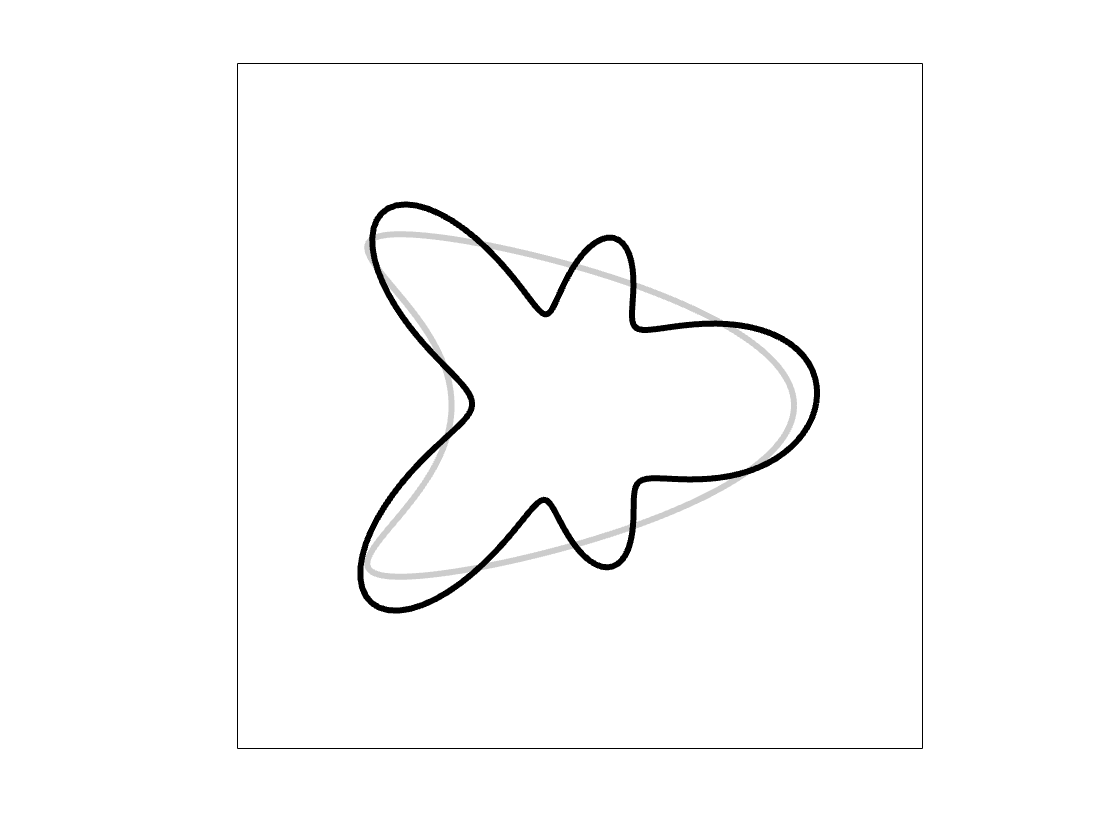}
\end{minipage}
\begin{minipage}{0.28\linewidth}
\includegraphics[width=0.98\linewidth,trim = 70 30 50 15, clip]{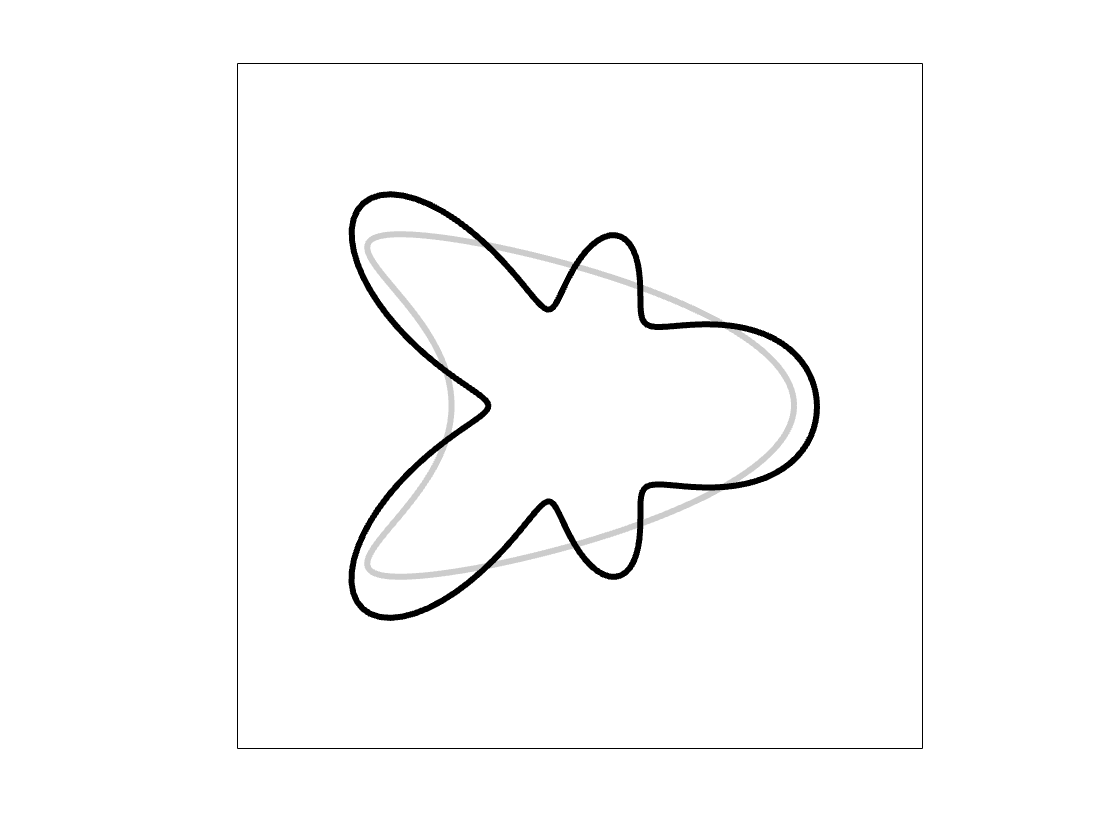}
\end{minipage}
\begin{minipage}{0.28\linewidth}
\includegraphics[width=0.98\linewidth,trim = 70 30 50 15, clip]{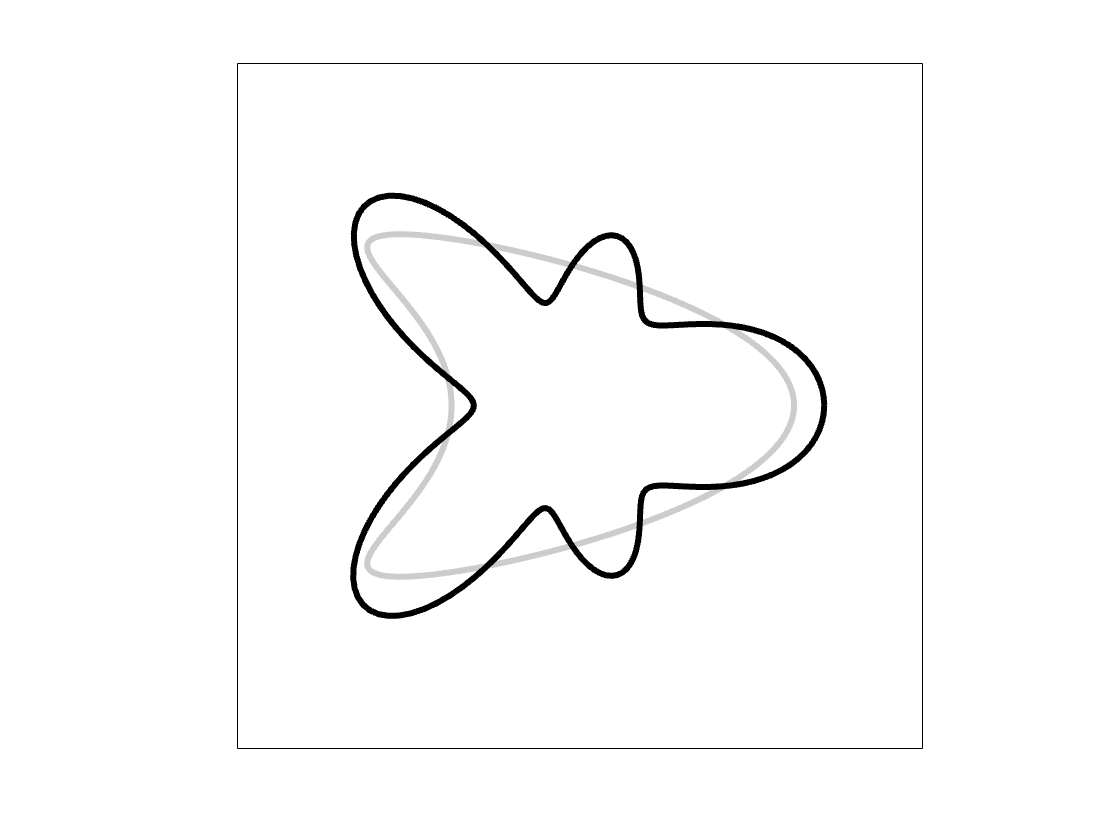}
\end{minipage}
\end{subfigure}
\begin{subfigure}{\linewidth}
\centering
\vskip -3mm
\captionsetup{justification=centering}
\begin{minipage}{0.1\linewidth}
\subcaption*{}
\end{minipage}
\begin{minipage}{0.28\linewidth}
\subcaption*{${a}_0^{\text{rec}}=0.8814 + 0.8146\rmi$\\${\gamma}^{\text{rec}}=1.0105$}
\end{minipage}
\begin{minipage}{0.28\linewidth}
\subcaption*{${a}_0^{\text{rec}}=0.9179 + 0.7986\rmi$\\${\gamma}^{\text{rec}}=0.9964$}
\end{minipage}
\begin{minipage}{0.28\linewidth}
\subcaption*{${a}_0^{\text{rec}}=0.8967 + 0.8001\rmi$\\${\gamma}^{\text{rec}}=1.0113$}
\end{minipage}
\end{subfigure}
\vskip -2mm
\caption{Recovery of a kite-shaped inclusion. Recovery with the EMTs up to higher-order terms provides finer details for the inclusion.}
\label{figure:Kite}
\end{figure}

\begin{figure}[ht!]
\begin{subfigure}{\linewidth}
\centering
\captionsetup{justification=centering}
\begin{minipage}{0.22\linewidth}
\includegraphics[width=1\linewidth,trim=70 30 50 15, clip]{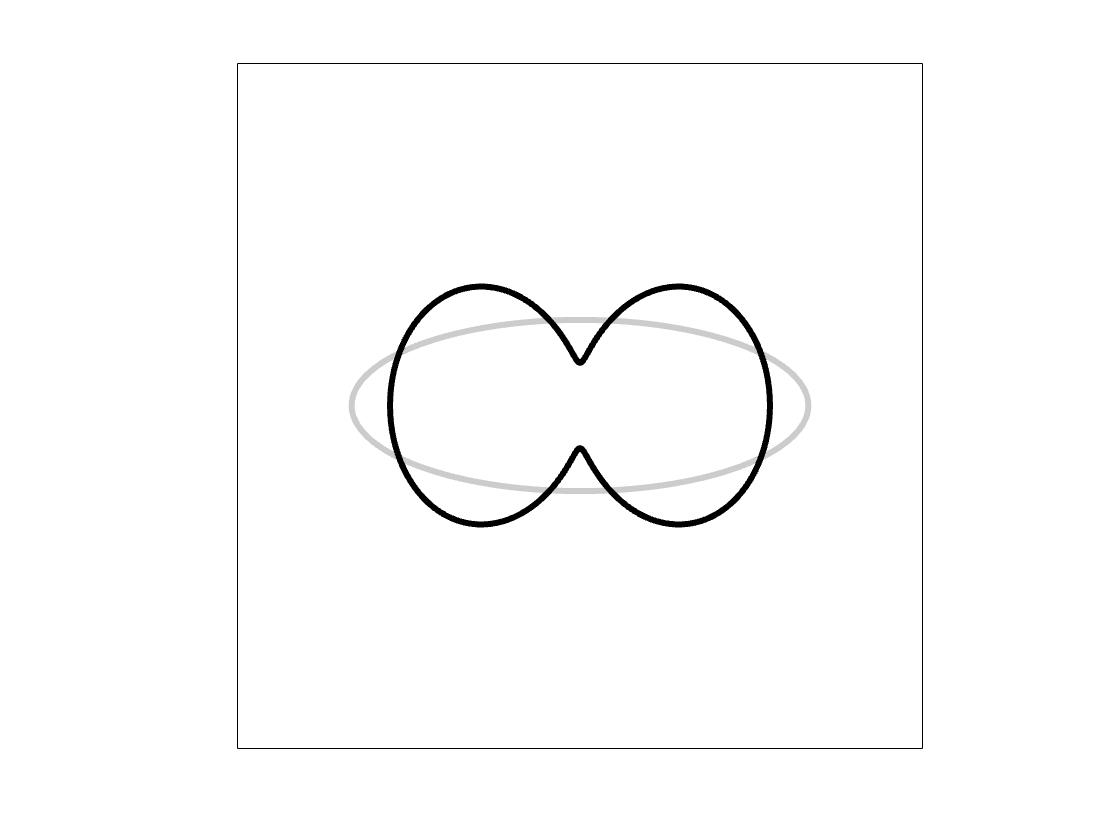}
\end{minipage}
\hskip -5.45mm
\begin{minipage}{0.22\linewidth}
\includegraphics[width=1\linewidth,trim=70 30 50 15, clip]{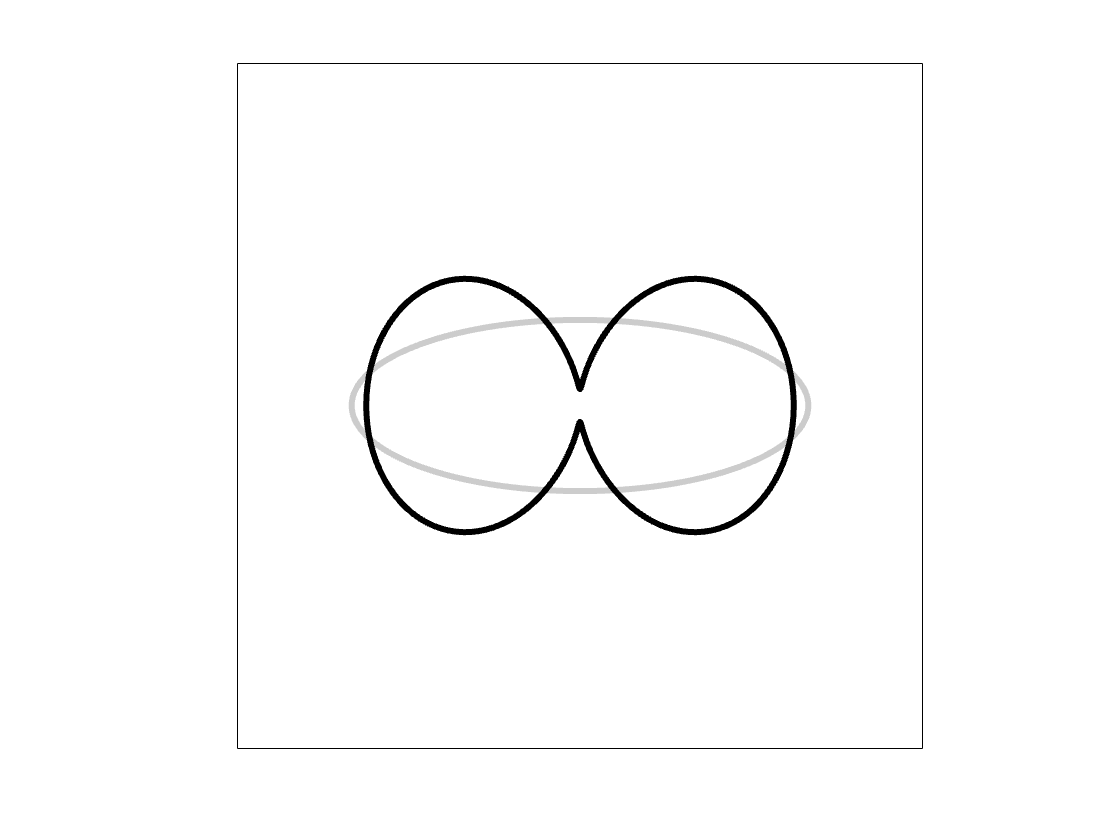}
\end{minipage}
\hskip -5.45mm
\begin{minipage}{0.22\linewidth}
\includegraphics[width=1\linewidth,trim=70 30 50 15, clip]{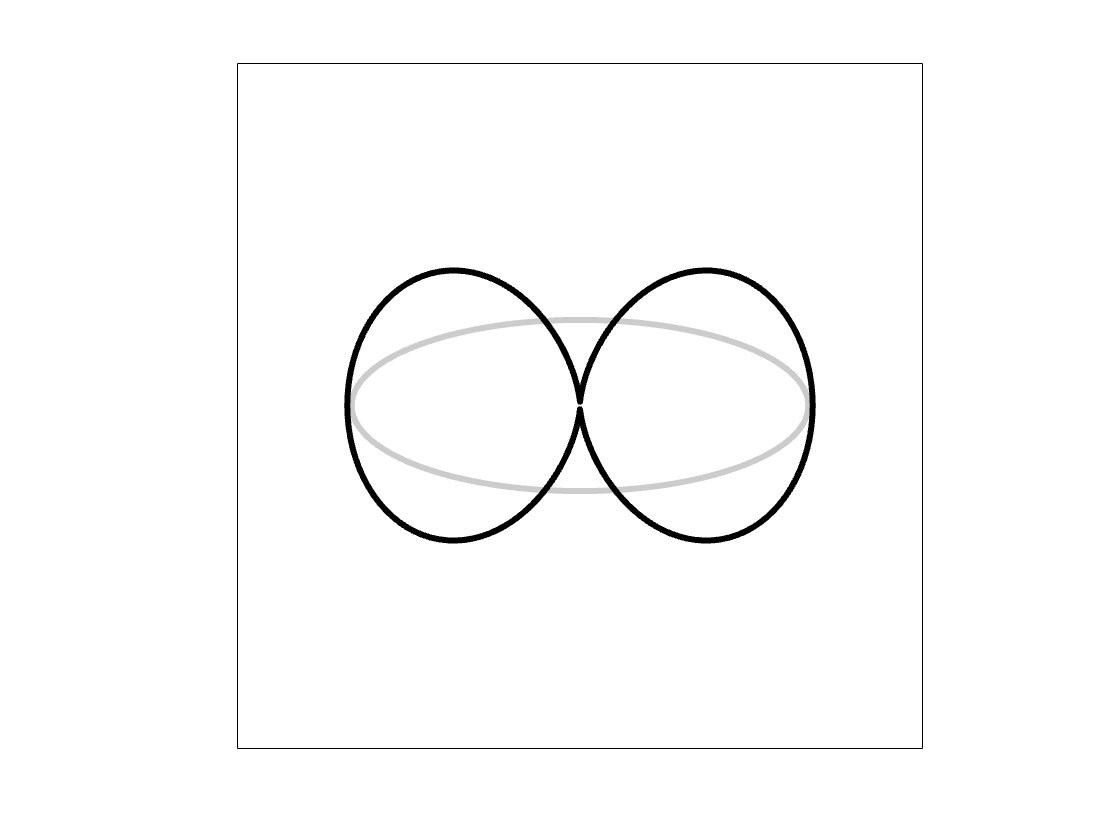}
\end{minipage}
\hskip -5.45mm
\begin{minipage}{0.22\linewidth}
\includegraphics[width=1\linewidth,trim = 70 30 50 15, clip]{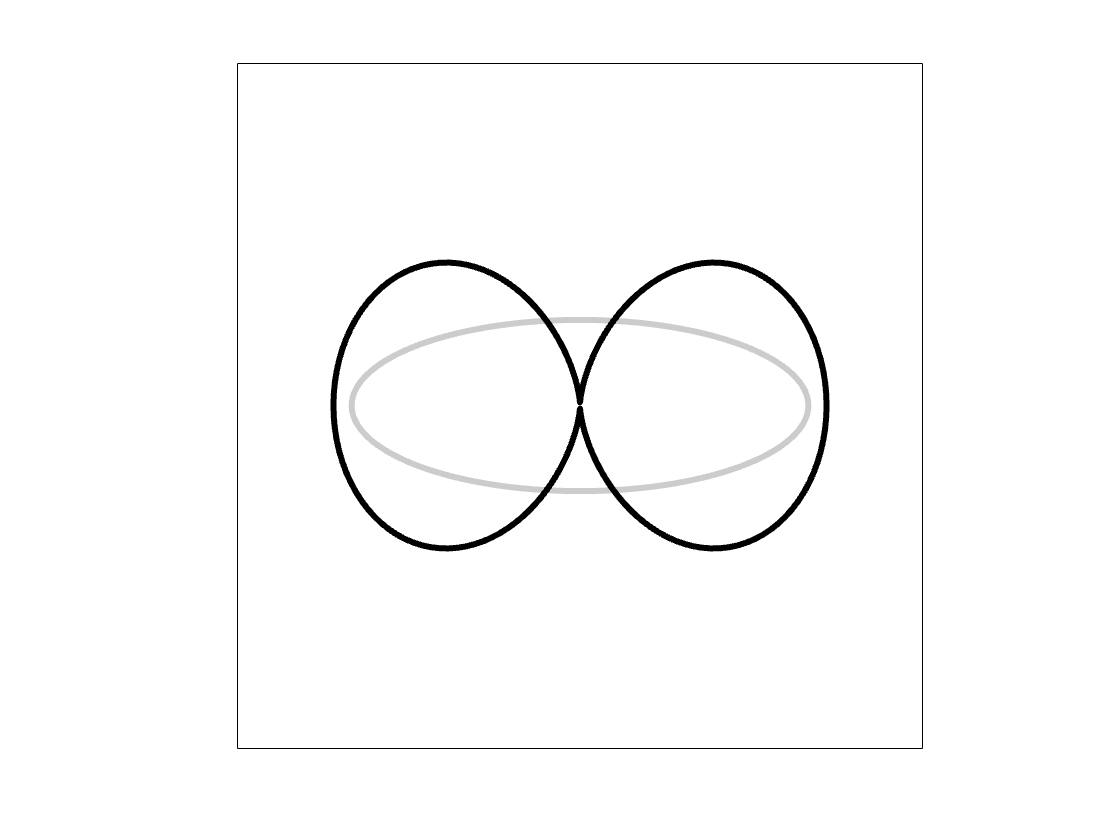}
\end{minipage}
\hskip -5.45mm
\begin{minipage}{0.22\linewidth}
\includegraphics[width=1\linewidth,trim = 70 30 50 15, clip]{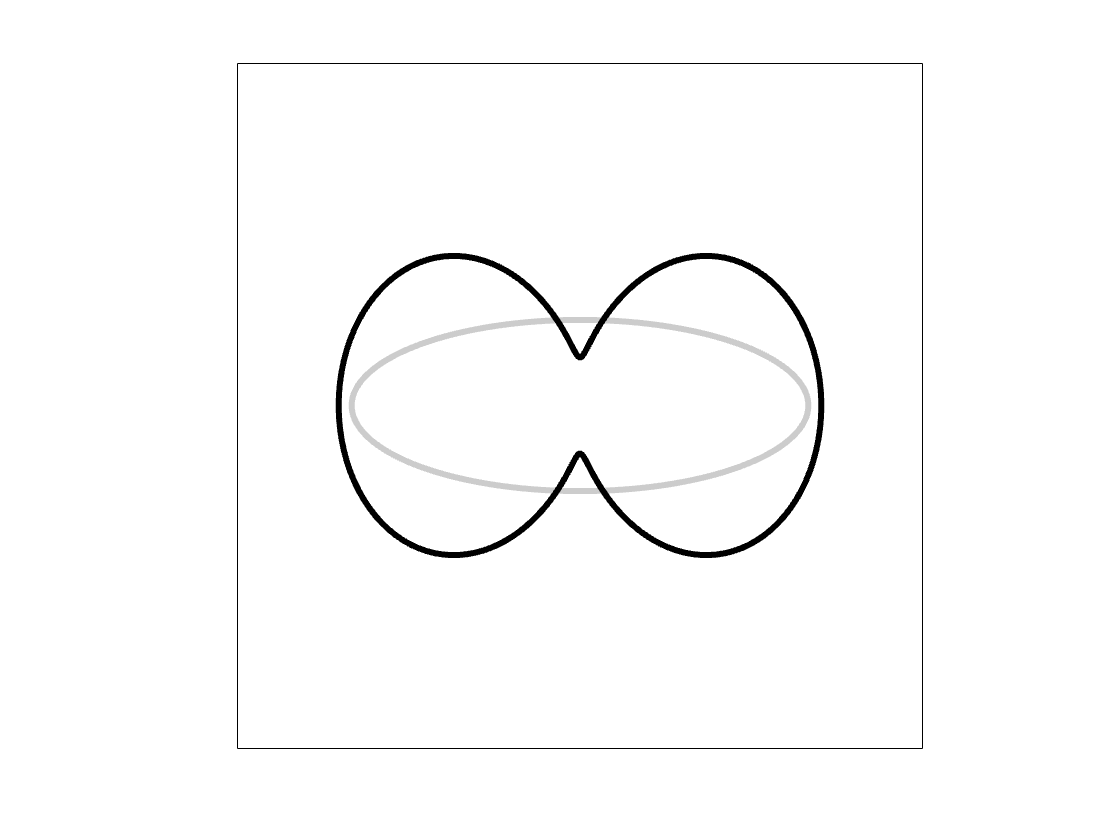}
\end{minipage}
\end{subfigure}
\begin{subfigure}{\linewidth}
\centering
\vskip -3mm
\captionsetup{justification=centering}
\begin{minipage}{0.22\linewidth}
\subcaption*{$\tilde{\lambda}=8 $\\ $\tilde{\mu}=5 $}
\end{minipage}
\hskip -5.45mm
\begin{minipage}{0.22\linewidth}
\subcaption*{$\tilde{\lambda}=1.8 $\\ $\tilde{\mu}=1.5 $}
\end{minipage}
\hskip -5.45mm
\begin{minipage}{0.22\linewidth}
\subcaption*{$\tilde{\lambda}=1 $\\ $\tilde{\mu}=0.8 $}
\end{minipage}
\hskip -5.45mm
\begin{minipage}{0.22\linewidth}
\subcaption*{$\tilde{\lambda}=0.6 $\\ $\tilde{\mu}=0.4 $}
\end{minipage}
\hskip -5.45mm
\begin{minipage}{0.22\linewidth}
\subcaption*{$\tilde{\lambda}=0.1 $\\ $\tilde{\mu}=0.08 $}
\end{minipage}
\end{subfigure}
\vskip -2mm
\caption{Recovery of an ellipse-shaped inclusion with $\text{Ord}= 4$ without noise.  
The reconstruction performance varies depending on $(\tlambda,\tmu)$ even when $\Om$ has the same shape. 
}
\label{figure:Ellipse}
\end{figure}

\section{Conclusion}\label{sec:conclusion}
We investigated the inverse problem of reconstructing a planar isotropic elastic inclusion embedded in a homogeneous background medium from the elastic moment tensors (EMTs). Utilizing the layer potential approach and the complex formulation for the plane elastostatic problem, we successfully obtained explicit formulas for the shape derivative of the contracted EMTs by considering the inclusion as a perturbed disk. We then derived explicit expressions in terms of the EMTs for the Fourier coefficients of the shape deformation function, which describes the shape of the inclusion. Our approach provides an analytic non-iterative reconstruction algorithm for the elastic inclusion. 

It is worth noting that the complex-variable formulation can also be applied to the elastostatic inclusion problem involving an elliptic inclusion \cite{Ando:2018:SPN,Mattei:2021:EAS} (see also \cite{Ammari:2027:BIM,Choi:2021:ASR}). In our future work, we plan to generalize \cref{thm:main_result} to address the analytic recovery of a perturbed ellipse. Additional questions remain to be investigated, such as developing analytic recovery formulas to simultaneously determine the geometry and material parameters $\tlambda,\tmu$ of the inclusion $\Om$, reducing the boundary regularity assumption on the inclusion, and deriving error estimates for the proposed reconstruction method.


%
\bibliographystyle{plain}

\end{document}